\documentclass[11pt,a4paper]{article}

\usepackage[utf8]{inputenc} 
\usepackage{url}

\usepackage{amsmath,amsfonts,amssymb}
\numberwithin{equation}{section}
\usepackage{amsthm}
\usepackage{graphicx}
\usepackage{enumerate}
\usepackage{multirow,bigdelim}
\usepackage{mathtools}
\usepackage{soul} 
\usepackage{pstricks}
\usepackage[ruled, boxed,  linesnumbered, norelsize]{algorithm2e}
\usepackage[margin=3cm]{geometry}
\usepackage[colorinlistoftodos,prependcaption,textsize=footnotesize]{todonotes}
\usepackage{array}
\usepackage[normalem]{ulem}
\usepackage{url}
\usepackage{hyperref}
\usepackage{enumerate}
\usepackage[shortcuts]{extdash}
\usepackage{mathtools}
\hypersetup{
	colorlinks,
	linkcolor={red!50!black},
	citecolor={blue!50!black},
	urlcolor={blue!80!black}
}
\graphicspath{{./pics/}}

\usepackage{tabularx}
\usepackage{makecell}
\usepackage{pifont}
\newcommand{\cmark}{\ding{51}}%
\newcommand{\xmark}{\ding{55}}%
\newcommand{\qmark}{\textbf{?}}%
\newcolumntype{Y}{>{\centering\arraybackslash}X}
\newcommand{\eqrefit}[1]{{$({\ref{#1}})$}}

\newcommand{\reCone}{\mathrm{rec}\,}
\newcommand{\lineality}{\mathrm{lin}\,}

\newcommand{\closure}{\mathrm{cl}\,}

\newcommand{\norm}[1]{\lVert{#1}\rVert}
\newcommand{\inProd}[2]{\langle #1 , #2 \rangle }

\newcommand{\PSDcone}[1]{{\mathcal{S}^{#1}_+}}	
	
\newcommand{\doubly}[1]{{\mathcal{D}^{#1}}}

\newcommand{\minFaceset}[2]{ {\text{minFace}(#1,#2)}}
\newcommand{\stdMap}{ {\mathcal{A}}}

\newcommand{\stdCone}{ {\mathcal{K}}}

\newcommand{\stdAffine}{ \mathcal{V}}
\newcommand{\stdFace}{ \mathcal{F}}
\newcommand{\stdFaceC}{F}

\newcommand{\matRange}{{\mathrm{ range } \,}}

\newcommand{\ambSpace}{\mathcal{E}}

\renewcommand{\Re}{\mathbb{R}}    
  
\renewcommand{\S}{\mathcal{S}}

\newcommand{\face}{\mathrel{\unlhd}}

\DeclareMathOperator{\lspan}{span}
\DeclareMathOperator{\aff}{aff}

\DeclareMathOperator{\cone}{cone}
\DeclareMathOperator{\conv}{conv}
\DeclareMathOperator{\dist}{dist}
\DeclareMathOperator{\ext}{ext}
\DeclareMathOperator{\interior}{int}
\DeclareMathOperator{\reInt}{ri}
\DeclareMathOperator*{\Argmax}{arg\ max}
\newcommand{\RR}{\mathbb{R}}
\newcommand{\R}{\mathbb{R}}

\theoremstyle{definition}

\newtheorem{definition}{Definition}[section]

\newtheorem{example}[definition]{Example}

\theoremstyle{theorem}

\newtheorem{lemma}[definition]{Lemma}
\newtheorem{proposition}[definition]{Proposition}

\newtheorem{corollary}[definition]{Corollary}

\newtheorem{theorem}[definition]{Theorem}
\newtheorem*{proposition*}{Proposition}

\theoremstyle{remark}
\newtheorem{remark}[definition]{Remark}

\title{Amenable cones are particularly nice}
\author{Bruno F.\ Louren\c{c}o \and Vera Roshchina \and James Saunderson}

\begin{document}
\maketitle

\begin{abstract} 
Amenability is a geometric property of convex cones that is stronger than facial exposedness and assists in the study of error bounds for conic feasibility problems. In this paper we establish numerous properties of amenable cones, and investigate the relationships between amenability and other properties of convex cones, such as niceness and projectional exposure. 

We show that the amenability of a compact slice of a closed convex cone is equivalent to the amenability of the cone, and prove several results on the preservation of amenability under intersections and other convex operations. It then follows that homogeneous, doubly nonnegative and other cones that can be represented as slices of the cone of positive semidefinite matrices are amenable.

It is known that projectionally exposed cones are amenable and that amenable cones are nice, however the converse statements have been open questions. We construct an example of a four-dimensional cone that is nice but not amenable. We also show that amenable cones are projectionally exposed in dimensions up to and including four. 

We conclude with a discussion on open problems related to facial structure of convex sets that we came across in the course of this work, but were not able to fully resolve.
\end{abstract}


\section{Introduction}

Amenability was introduced in \cite{L19} in the context of error bounds for convex cones. In particular, consider the  following conic feasibility problem
\begin{equation}\label{eq:cfp}
\mathrm{ find }\quad x \in \stdCone \cap \mathcal{V}, \tag{CFP}
\end{equation}
where $\stdCone$ is a closed convex cone and $\mathcal{V}$ is an affine subspace. 
If $\stdCone$ is an \emph{amenable cone}, there are a number of techniques 
that simplify the study of error bounds for the system \eqref{eq:cfp}, especially when the goal is to obtain bounds that hold without constraint qualifications, see \cite{L19}.

Given the ubiquity and the usefulness of error bounds throughout optimization (see, e.g., \cite{Pang97,LP98}), it is natural to try to develop our understanding of amenability. In this work we extend the notion of amenable cones to arbitrary convex sets. Doing so allows us to show that the intersection of amenable sets is amenable and that all affine slices of an amenable cone must be amenable. Conversely, if a cone is generated by a compact amenable slice, it must be amenable. 

Amenability is a stronger form of \emph{facial exposedness}, which is a notion that goes back at least to the 1930s \cite{St35}. There are several other ways to strengthen the classical notion of facial exposure that are commonly used in the literature. The notion of \emph{niceness} (facial dual completeness) has its origins in optimality conditions for general conic convex optimization problems and in the facial reduction algorithm of Borwein and Wolkowicz \cite{BW81}, see Remarks~6.1 and 6.2 therein. (The name \emph{nice} itself seems to have appeared later.) Pataki has shown that nice cones admit extended duals that fix certain theoretical issues related to classical Lagrangian duality \cite{Pa13_2}. Niceness also features in results on when a linear image of a dual of a convex cone is closed \cite{LiuPataki,PatakiLI} and in the study of conic lifts of convex sets: when a cone is nice, certain results related to lifts can be sharpened, see \cite[Corollary~1]{GPR13}. 
Pataki showed in \cite{Pa13_2} that nice cones are always facially exposed and conjectured that the converse was true. This was disproved in \cite{Vera}, where a four-dimensional cone that is facially exposed but not nice is constructed. Niceness also appears to have a direct relation to error bounds: necessary and sufficient conditions for niceness were obtained using subtransversality-like tangential relations in \cite{RT}. It was shown in \cite{L19} that amenable cones are nice. In this paper we show that nice cones are not always amenable.

Another notion that we pay close attention to in this paper is \emph{projectional exposedness}, which goes back to \cite{BW81}, also in connection to optimality conditions for conic convex optimization problems and the so-called \emph{facial reduction algorithm}. See also \cite{BLP87,PL88,ST90}. It was shown in \cite{L19} that projectionally exposed cones are amenable. In this paper we show that the converse is true in dimensions up to and including four. In particular, if there exists an amenable cone that is not projectionally exposed, it must have dimension at least five.

Finally, we show that homogeneous and doubly nonnegative cones are amenable, in particular generalising the previously known result for symmetric cones \cite{L19}. This contributes to the evidence that amenability is a valuable notion that captures the benign properties of many important classes of structured cones.

This paper is organised as follows. Section~\ref{sec:prel}
contains preliminaries on the facial structure of convex sets and cones: we state and provide references for known technical results that are used throughout the paper. 

In Section~\ref{sec:amenable} we discuss basic properties of amenable sets. 
In Section~\ref{ssec:defbounded} we extend the definition of amenability from cones to general convex sets, discussing the subtleties related to non-compactness that are absent in the conic setting. (Specifically, see Example~\ref{ex:am_bd} based on a geometric construction from \cite{Sturm00}). We highlight the motivation via subtransversality, proving that amenability of a face is equivalent to subtransversality of the affine span of $F$ and the set $C$ in  Proposition~\ref{prop:am_eq}. We also demonstrate that amenability is preserved under some common convex operations, such as intersections and direct products (see Proposition~\ref{prop:am_int}).

Section~\ref{sec:slices} is dedicated to showing that amenability of a cone is equivalent to the amenability of its compact base. Studying amenability of slices often makes the geometry more intuitive, and reduces the dimension of the problem.

In Section~\ref{sec:nicenotamen} we construct an example of a cone that is amenable but not nice (facially dual complete).

Section~\ref{sec:amenproj} is dedicated to the relationship between amenability and projectional exposure. We prove that for amenable cones, faces of codimension one are projectionally exposed (see Theorem~\ref{thm:co_proj}), and this allows us to conclude that all amenable cones in spaces of dimension at most 4 are projectionally exposed (Corollary~\ref{cor:proj}).


In the last section we state open questions related to facial structure of convex sets and provide additional insights.

%
\section{Preliminaries}\label{sec:prel}
Here we recall some facts about convex sets and their faces. We let $\ambSpace$ denote some finite dimensional Euclidean space equipped with an 
inner product $\inProd{\cdot}{\cdot}$ and an induced norm 
$\norm{\cdot}$.
Let $C\subseteq \ambSpace$ be a convex set. We denote its closure, relative interior, interior,
affine hull, dimension, span and orthogonal complement by $\closure C, \reInt C, \interior C, \aff C, \dim{C}, \lspan C, C^\perp$, respectively.
The recession cone of $C$ is denoted by $\reCone C$ and its lineality space by $\lineality C$, so that 
$\lineality C = \reCone C \cap (- \reCone C)$.
We denote by $\cone C$ the cone generated by $C$, i.e., 
\[
\cone C \coloneqq \{\lambda x \mid x \in C, \lambda \geq 0 \}.
\]
Given $x\in \ambSpace$, we define the distance from $x$ to $C$ as 
\begin{equation}\label{eq:def_dist}
\dist(x,C) \coloneqq \inf \{\norm{x-y} \mid y \in C \}.
\end{equation}
If $U \subseteq \ambSpace$ is an arbitrary subset, we denote by $\conv U$ the convex hull of $U$.

Throughout the paper we adopt the following convention.
We will use $C,\stdFaceC$ for convex sets and their faces, respectively. $\stdCone, \stdFace$ will be used for convex cones and their faces, respectively. We denote by $\S^n$ the space of $n\times n$ real symmetric matrices and by $\PSDcone{n}$ the cone of $n\times n$ real symmetric positive semidefinite matrices.

\subsection{On faces of convex sets}
Here, we collect a few results and facts on faces of convex sets 
that will be useful in later sections.
First we recall that a closed convex set $\stdFaceC$ contained in $C$ is said to be a \emph{face} if  whenever $x,y \in C$ are such that 
$ \alpha x + (1-\alpha) y \in \stdFaceC $ for some $\alpha \in (0,1)$, we have $x,y \in \stdFaceC$.
In this case, we write $\stdFaceC \face C$. Faces consisting of a single point are called \emph{extreme points} and the set of extreme points of $C$ will be denoted by $\ext C$.
A face $\stdFaceC \face C$ is said to be \emph{proper} if $\stdFaceC \neq C$.
Given some convex subset $S \subseteq C$ we denote 
by $\minFaceset{S}{C}$ the \emph{minimal face of $C$ containing $S$}. For $\stdFaceC \face C$, we have the following characterization of the minimal face:
\begin{equation}\label{eq:min_face}
\stdFaceC = \minFaceset{S}{C} \quad \Longleftrightarrow \quad  \reInt(S) \cap \reInt (\stdFaceC) \neq \emptyset \quad \Longleftrightarrow \quad \reInt(S) \subseteq \reInt \stdFaceC,
\end{equation}
i.e., the minimal face of $C$ containing $S$ is the unique face such that the relative interior of $S$ intersects the relative interior of $\stdFaceC$. For 
the first implication, see \cite[Proposition~3.2.2]{Pa00}. The second implication 
follows because $\reInt (S) = \reInt( S \cap \stdFaceC) =  \reInt(S) \cap \reInt (\stdFaceC) \subseteq \reInt (\stdFaceC)$ holds when $\reInt(S) \cap \reInt (\stdFaceC) \neq \emptyset$ and $S \subseteq \stdFaceC$ (see \cite[Theorem~6.5]{Ro97}).

A face $\stdFaceC \face C$ is said to be \emph{facially exposed} if there exists a supporting hyperplane $H$ of $C$ such that $\stdFaceC  = C \cap H$. The following result on exposed faces is 
well-known but we give a short proof, see also 
\cite[Lemma~2.3]{BM94} for a related result.
\begin{proposition}[Every proper face is contained in some proper exposed face]\label{prop:exp_face}
Let $\stdFaceC \face C$ be such that $\stdFaceC \neq C$.
Then, there exists an exposed face $\stdFaceC' \face C$ satisfying $\stdFaceC \subseteq \stdFaceC'$ and $\stdFaceC' \neq C$.
\end{proposition}
\begin{proof}
	Because $\stdFaceC \neq C$, we must have $\reInt (\stdFaceC) \cap \reInt (C) = \emptyset$ (see, e.g.,~\cite[Corollary~18.1.2]{Ro97}).
Then $\stdFaceC$ and $C$ can be \emph{properly separated}, i.e., there exists a hyperplane $H$ such that $\stdFaceC$ and $C$ belong to opposite closed half-spaces defined by $H$ and at least one among $\stdFaceC$ and $C$ is not entirely contained in $H$ (see \cite[Theorem~11.3]{Ro97}). Because $\stdFaceC \subseteq C$, it must be the case that $\stdFaceC \subseteq H$ and that $H$ is a supporting hyperplane of $C$.
Since the separation is proper, there exists at least one point of $C$ not in $H$. 
Therefore, the exposed face $\stdFaceC' \coloneqq C \cap H$ satisfies $\stdFaceC' \neq C$ and $\stdFaceC \subseteq \stdFaceC'$.
\end{proof}

The next proposition is contained in the results of Section~IV of \cite{Du62}, but for self-containment sake, we give a short argument.
\begin{proposition}\label{prop:face_intersection}
Let $C_1,C_2$ be closed convex sets such that 
$C \coloneqq C_1 \cap C_2$ is non-empty.
Let $\stdFaceC \face C$. Then, there are 
$\stdFaceC_1 \face C_1$, $\stdFaceC_2 \face C_2$ such that \[
\stdFaceC = \stdFaceC_1 \cap \stdFaceC_2,\qquad 
\reInt(\stdFaceC) = \reInt(\stdFaceC_1) \cap \reInt (\stdFaceC_2).
\]
\end{proposition}
\begin{proof}
Let $\stdFaceC_1 \coloneqq \minFaceset{F}{C_1}$ and $\stdFaceC_2 \coloneqq \minFaceset{F}{C_2}$. By \eqref{eq:min_face}, we have
\[
\reInt(F) \subseteq \reInt (\stdFaceC_1) \cap \reInt (\stdFaceC_2).
\] 
In particular, $\stdFaceC_1$ and $\stdFaceC_2$ have a relative interior point in common, so, we have
$\reInt(\stdFaceC_1) \cap \reInt (\stdFaceC_2) = \reInt(\stdFaceC_1\cap \stdFaceC_2)$, see \cite[Theorem~6.5]{Ro97}.
Therefore,
\begin{equation}\label{eq:min_face_aux}
\reInt(F) \subseteq \reInt (\stdFaceC_1\cap \stdFaceC_2) =  \reInt (\stdFaceC_1) \cap \reInt (\stdFaceC_2)
\end{equation}
Because $\stdFaceC_1 \face C_1$ and 
$\stdFaceC_2 \face C_2$, we have 
that $\stdFaceC_1 \cap \stdFaceC_2$ is a face of $C$.
Since $\stdFaceC$ is also a face of $C$, \eqref{eq:min_face_aux} implies that $\stdFaceC = \stdFaceC_1 \cap \stdFaceC_2$.
\end{proof}

\subsection{Cones and notions of facial exposedness}
First, we recall that a closed convex cone $\stdCone \subseteq \ambSpace$ is said to be \emph{pointed} if 
$\lineality \stdCone = \{0\}$ and \emph{full-dimensional} if $\dim{\stdCone} = \dim \ambSpace$. A face $\stdFace \face \stdCone$ such that
$\dim \stdFace = 1$ is called an \emph{extreme ray}. 
If $\stdFace = \{\alpha x \mid \alpha \geq 0 \}$ we say that 
$\stdFace$ is \emph{generated by $x$}.

Here, we recall some properties stronger than facial exposedness for cones. 
We say that a cone $\stdCone$ is \emph{nice} (or \emph{facially dual complete}) if for every face $\stdFace \face \stdCone$ we have 
\[
\stdFace^* = \stdCone^* + \stdFace^\perp
\]
where $\stdCone^*$ is the dual cone of $\stdCone$, consisting of all linear functionals on $\ambSpace$
that take nonnegative values on $\stdCone$. Equivalently, we have that $\stdCone^* + \stdFace^\perp$ is closed for all 
$\stdFace \face \stdCone$.
%
%
A face $\stdFace \face \stdCone$ is said to \emph{projectionally exposed} if there exists an 
idempotent linear map $\mathcal{P}: \ambSpace \to \ambSpace$ (i.e., a linear projection that is not necessarily orthogonal) such that
\[
\mathcal{P} (\stdCone) = \stdFace.
\]
$\stdCone$ is said to be projectionally exposed if  every face is projectionally exposed.

Finally, $\stdCone$ is said to be \emph{amenable}, if for every face $\stdFace \face \stdCone$ there exists a constant $\kappa  > 0$ (possibly depending on $\stdFace$) such that
\begin{equation}\label{eq:def_am_cone}
\dist(x, \stdFace) \leq \kappa \dist(x, \stdCone), \quad \forall x \in \lspan \stdFace.
\end{equation}
Gathering several results in the literature we have the following.

\begin{proposition}[Notions of exposedness]\label{prop:imp}
	Consider the following statements.
	\begin{enumerate}[$(i)$]
		\item $\stdCone$ is projectionally exposed.
		\item $\stdCone$ is amenable.
		\item $\stdCone$ is nice
		\item $\stdCone$ is facially exposed.
	\end{enumerate}
	Then $(i)\Rightarrow (ii) \Rightarrow (iii)\Rightarrow (iv)$. If $\dim{\stdCone} \leq 3$ then $(iv) \Rightarrow (i)$.
\end{proposition}
\begin{proof}
The implication 	$(i)\Rightarrow (ii) \Rightarrow (iii)$ follows from Propositions~9 and 13 in \cite{L19}.
The implication $(iii) \Rightarrow (iv)$  comes from \cite[Theorem~3]{pataki}.

Finally,  Poole and Laidacker proved that when 
 $\dim{\stdCone} \leq 3$, facial exposedness implies projectional exposedness \cite[Theorem~3.2]{PL88}.
\end{proof}

To conclude this subsection, we comment briefly on some applications of the notion of amenability. In \cite{L19}, the author describes how to compute {error bounds} for amenable cones. This computation relies on obtaining the so-called \emph{facial residual functions} (FRFs) and combining FRFs with the facial reduction algorithm \cite{BW81}, see also \cite{LLP20}. 

Error bounds themselves are important tools for the analysis of optimization problems \cite{Pang97}. In particular, the behavior of several algorithms can be described by the kind of error bound that holds between the underlying sets see, for example, \cite{BLT17}. For a discussion on convergence analysis of algorithms in the context of amenable cones and connections to the notion of \emph{singularity degree}, see \cite{LL20}.

Finally, as we show that certain classes of cones are amenable, (non-)amenability then becomes a reasonable criterion for proving that a given cone \emph{does not} belong to some target class. Because amenability implies facial exposedness and niceness, non-amenability is more likely to work as a witness of non-membership. 
For example, we will show in Corollary~\ref{cor:am} that spectrahedral sets are amenable. In particular, our example of a nice but not amenable cone described in Section~\ref{sec:nicenotamen} is not spectrahedral. We believe this would be nontrivial to establish using other methods.
\subsection{On bounded linear regularity}
We say that convex sets $C_1,\ldots, C_{m} \subseteq \ambSpace$ 
satisfy \emph{bounded linear regularity} if 
their intersection $C \coloneqq \bigcap _{i=1}^m C_i$ is nonempty and 
the following error bound condition holds: for every bounded set $B \subseteq \ambSpace$, there exists $\kappa_B > 0$ such that 
\begin{equation}\label{eq:blr}
\dist(x, C) \leq \kappa_B \max _{1 \leq i \leq m} \dist(x,C_i), \qquad \forall x \in B.
\end{equation}
Bounded linear regularity coincides with the notion of bounded $1$-H\"older regularity, see, for example, \cite[Definition~2.2]{BLT17} and the comments afterwards.
In the next sections we will need the following result, see \cite[Corollary~3]{BBL99} for a proof.
\begin{proposition}\label{prop:blr}
Let	$C_1,\ldots, C_{m} \subseteq \ambSpace$ 
be such that  $C_1,\ldots, C_{k}$ are polyhedral sets 
and
\[
\left(\bigcap_{i = 1}^{k} C_i\right) \bigcap \left(\bigcap_{j = k+1}^{m} \reInt C_j\right) \neq \emptyset.
\]	
holds. Then, $C_1,\ldots, C_{m}$ satisfy 
bounded linear regularity.
\end{proposition}
We mention in passing that other sufficient criteria for bounded linear regularity can be seen in \cite{BBL99} and in \cite[Theorem~7]{BD05}.

\section{Amenable convex sets  and their basic properties}\label{sec:amenable}
Amenability was originally defined for cones only, as in \eqref{eq:def_am_cone}.
Our first task is to extend this definition to arbitrary convex sets. There are two main motivations for that. The first is that the facial structure of convex sets is also an important subject on its own. The second is that when analyzing the properties of a convex cone,  it can be more convenient to analyze its \emph{slices} first, because they are lower dimensional objects. In fact, in Section~\ref{sec:slices} we will show that the amenability of a pointed closed convex cone is equivalent to the amenability of its slices, see Proposition~\ref{prop:slices} and Theorem~\ref{thm:slice}.

\subsection{Definition of amenability for general convex sets}\label{ssec:defbounded}

Let $C$ be an arbitrary convex set and let $\stdFaceC \face C$ be a face.
With that, we have
\[
\stdFaceC = C \cap \aff \stdFaceC.
\] 
A first attempt at extending amenability \eqref{eq:def_am_cone} to general closed convex sets would be to require the existence 
of some $\kappa > 0$ such that 
\begin{equation}\label{eq:def_am_c_bad}
\dist(x, \stdFaceC) \leq \kappa \dist(x, C), \quad \forall x \in \aff \stdFaceC.
\end{equation}
Unfortunately, this is unlikely to hold for many reasonable sets, as we will 
see in Example~\ref{ex:am_bd}. The key is to restrict the validity of  \eqref{eq:def_am_c_bad} to bounded sets 
as follows.

\begin{definition}[Amenable faces and amenable sets]\label{def:am}
	Let $C$ be a closed convex set and $\stdFaceC \face C$ be a face.  $\stdFaceC$ is said to be \emph{amenable} if 
	for every bounded set $B$, there exists a constant $\kappa > 0$ (possibly depending on $F$ and $B$) such that 
	\begin{equation}\label{eq:am_def}
	\dist(x, \stdFaceC) \leq \kappa \dist(x, C), \quad \forall x \in (\aff \stdFaceC )\cap B.
	\end{equation}
	If all faces of $C$ are amenable, then $C$ is said to be an \emph{amenable convex set}.
\end{definition}
Next, we advance the case that Definition~\ref{def:am} is reasonable by presenting a few equivalences. Recall that $C_1$ and $C_2$ are \emph{subtransversal at $x^* \in C_1\cap C_2$} \cite[Definition~7.5]{ioffe} if there is 
a neighbourhood $U$ of $x^*$ and $\kappa > 0$ such that 
\begin{equation}\label{eq:am_sub}
\dist(x, C_1 \cap C_2) \leq \kappa (\dist(x, C_1)+\dist(x,C_2)), \quad \forall x \in U. 
\end{equation}

\begin{proposition}\label{prop:am_eq}
	Let $C$ be a convex set and let $\stdFaceC \face C$ be a face. The following are equivalent:
	\begin{enumerate}[$(i)$]
		\item $\stdFaceC$ is an amenable face of $C$. \item $C$ and $\aff \stdFaceC$ are boundedly linearly regular, i.e., for every bounded set $B$ there exists $\kappa _B > 0$ such that 
		\begin{equation}\label{eq:am_blr}
		\dist(x,\stdFaceC) \leq \kappa _B \max \{\dist(x,\aff \stdFaceC), \dist(x,C) \},\qquad \forall x \in B.
		\end{equation}
		\item $C$ and $\aff \stdFaceC$ are subtransversal  at every point of $\stdFaceC$.
	\end{enumerate} 
\end{proposition}
\begin{proof}
	First, we observe that \eqref{eq:am_def} is implied by \eqref{eq:am_blr}, when 
	$C_1 = C$, $C_2 = \aff \stdFaceC$ and $x \in \aff \stdFaceC$. Therefore
	\fbox{$(i) \Leftarrow (ii)$} holds.
	
	We move on to proving that \fbox{$(i) \Rightarrow (ii)$}. 
	Suppose $\stdFaceC$ is amenable, let $B$ be an arbitrary bounded set and let $P$ denote the  projection operator onto $\aff \stdFaceC$, i.e., $P(x) = \arg\min _{y\in \aff\stdFaceC}\norm{x-y}$ holds for every $x$.
	Since $P$ is nonexpansive, $P(B)$ must be bounded as well. 
	By the definition of amenability, there exists $\kappa$ such that 
	\begin{equation}\label{eq:equiv}
	\dist(z, \stdFaceC) \leq \kappa \dist(z, C), \quad \forall z \in (\aff \stdFaceC )\cap P(B).
	\end{equation}
	Given $x \in B$, using \eqref{eq:equiv} and the properties of the projection operator, we have 
	\begin{align*}
	\dist(x, \stdFaceC) &\leq \dist(x, \aff \stdFaceC) + \dist(P(x),\stdFaceC)\\
	&\leq \dist(x, \aff \stdFaceC) + \kappa \dist(P(x),C)\\
	&\leq \dist(x, \aff \stdFaceC) + \kappa (\dist(x,C) + \dist(x,\aff \stdFaceC))\\
	&\leq (\kappa+1)(\dist(x,C) + \dist(x,\aff \stdFaceC))\\
	&\leq 2(\kappa+1)\max \{\dist(x,C),\dist(x,\aff \stdFaceC)\}.
	\end{align*}
	This shows that $C$ and $\aff \stdFaceC$ are boundedly linearly regular.
	
	Next, we check that \fbox{$(ii) \Rightarrow (iii)$}. 
	Let $x^* \in \stdFaceC$ and let $U$ be any bounded neighbourhood of $x^*$. 
	Since $(ii)$ holds, there exists $\kappa > 0$ such that 
	\[
	\dist(x, \stdFaceC) \leq \kappa (\dist(x,C) + \dist(x,\aff \stdFaceC)), \quad \forall x \in U,
	\]
	which shows that $C$ and $\aff \stdFaceC$ are subtransversal  at $x^*$.
	
	Finally, we show that \fbox{$(ii) \Leftarrow (iii)$}. 
	Suppose that $\aff \stdFaceC$ and $C$ are subtransversal at every 
	$x \in \stdFaceC$. Let $B$ be a bounded set and denote by $\bar{{B}}$ its closure.
	For every $x \in \bar{{B}}$, there exists some open neighbourhood $U_x$ and a constant 
	$\kappa _x$ such that \eqref{eq:am_sub} holds. Since $\bar{{B}}$ is compact,
	and the $U_x$ form an open cover for  $\bar{{B}}$,
	there are finitely many $x_1,\ldots, x_\ell$ such that 
	\[
	\bar{{B}} \subseteq \bigcup _{i = 1}^{\ell} U_{x_i}.
	\]
	Therefore, if we set 
	\[
	\kappa = \max \{\kappa _{x_1},\ldots, \kappa_{x_\ell}\},
	\]
	then, for every $x \in B$, we have
	\[
	\dist(x, F)\leq \kappa (\dist(x,\aff \stdFaceC) + \dist(x,C)) \leq 2\kappa \max \{\dist(x,\aff \stdFaceC), \dist(x,C) \}.
	\]
\end{proof}
Let $\stdCone$ be a closed convex cone. In \cite[Proposition~12]{L19} it was shown that a face $\stdFace \face \stdCone$ satisfies \eqref{eq:def_am_cone} if and only if $\stdCone$ and $\lspan \stdFace$ are boundedly linearly regular. Since $\lspan \stdFace = \aff \stdFace$, in view of Proposition~\ref{prop:am_eq}, we conclude that $\stdCone$ is amenable as a cone (i.e., \eqref{eq:def_am_cone} is satisfied for every face) if and only if $\stdCone$ is amenable as a convex set (i.e., Definition~\ref{def:am} is satisfied).

Enforcing boundedness allows to prove closure of amenable sets under several common operations, such as intersections, direct products, linear transformations and lifts (see Proposition~\ref{prop:am_int} and Theorem~\ref{thm:slice}). The following example shows that boundedness is essential in Definition~\ref{def:am} when dealing with general convex sets.

\begin{example}[Boundedness is essential in the definition of amenability]\label{ex:am_bd}
In the definition of amenability, we require that \eqref{eq:am_def} holds
only when a bounded set $B$ is specified and $\kappa$ is allowed to change with $B$. Here, we show an example, based on Example~1 in \cite{Sturm00},  of an amenable convex set for 
which \eqref{eq:am_def} does not hold globally.
Let 
\begin{equation}\label{eq:boundedexample}
C \coloneqq \left\{\begin{pmatrix} x_{11} & x_{12} \\ x_{12} & x_{22} \end{pmatrix} \in \PSDcone{2} \mid x_{22} \geq 1 \right\}.
\end{equation}
The set $C$ is the intersection of an ice-cream cone and a half-space, shown in Fig.~\ref{fig:examplebounded}.
\begin{figure}[ht]
	\begin{center}
		\includegraphics[width=0.5\textwidth]{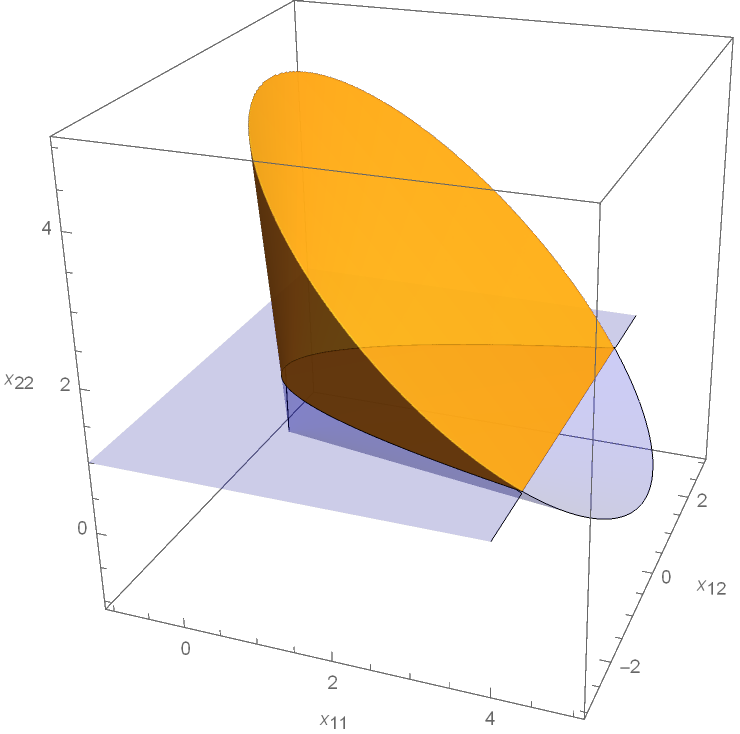}
	\end{center}
	\caption{A subset of the $2\times 2$ positive semidefinite cone described by \eqref{eq:boundedexample} which shows that  amenability must be considered locally, in general.}
	\label{fig:examplebounded}
\end{figure}
Second-order cones and half-spaces are amenable \cite{L19}, and it will be shown in Proposition~\ref{prop:am_int} that intersections of amenable sets are amenable, so $C$ is amenable.

Next we consider the following face of $C$
\[
\stdFaceC \coloneqq \left\{\begin{pmatrix} x_{11} & x_{12} \\ x_{12} & 1 \end{pmatrix} \in \PSDcone{2}  \right\}.
\]
$\stdFaceC$ is indeed a face of $C$ because it is obtained as an intersection of 
$C$ with the supporting hyperplane 
\[
H \coloneqq \left\{\begin{pmatrix} x_{11} & x_{12} \\ x_{12} & x_{22} \end{pmatrix} \in \S^2 \mid x_{22} = 1 \right\} = \{x \in \S^2 \mid \inProd{x}{d} = 1 \},
\]
where $d \in \PSDcone{2}$ is the matrix such that $d_{22} = 1$ and is zero elsewhere.
We note that the affine hull of $F$ is $H$.

Next,  we look at whether there could possibly exist some constant $\kappa > 0$  such that 
\begin{equation}\label{eq:bound_cd}
\dist(x, \stdFaceC) \leq \kappa \dist(x, C), \qquad \forall x \in \aff \stdFaceC,
\end{equation}
i.e., whether amenability could hold globally. For concreteness, we use the distance on $\S^2$ induced by the Frobenius norm. 
So suppose that \eqref{eq:bound_cd} holds. 
Similar to the sequence of inequalities in the proof that  $(i)\Rightarrow (ii)$ in Proposition~\ref{prop:am_eq}, we have 
\begin{align}
\dist(x, \stdFaceC) & \leq (\kappa+1)(\dist(x,C) + \dist(x,\aff \stdFaceC)),\quad \forall x \in \S^2. \label{eq:bound_cd2}
\end{align}
Next, we consider the following family of points indexed by $\epsilon > 0$:
\[
x^\epsilon \coloneqq \begin{pmatrix}
1/(\epsilon^2+\epsilon^3) & 1/\epsilon \\
1/\epsilon & 1 + \epsilon
\end{pmatrix}.
\]
We observe that $x^{\epsilon} \in C$, so $\dist(x^{\epsilon}, C) = 0$.
Furthermore $\dist(x^{\epsilon},\aff F) \leq \epsilon$.
Following essentially the same line of argument presented in Example~1 of \cite{Sturm00}, we will derive a contradiction as follows. Let $z^{\epsilon} =\arg \min _{z \in \stdFaceC} \norm{x^{\epsilon}-z}$ and let $y^{\epsilon} \coloneqq z^{\epsilon}-x^{\epsilon}$.
With that, we have $x^{\epsilon} + y^{\epsilon} \in F$  and \[\norm{y^{\epsilon}} = \dist(x^{\epsilon},F).\]
Using \eqref{eq:bound_cd2}, we also have 
\begin{equation}\label{eq:bound3}
\norm{y^{\epsilon}} = \dist(x^{\epsilon},F) \leq (\kappa+1)\epsilon.
\end{equation}
Since  $y^{\epsilon}_{22} + x^{\epsilon}_{22}  = 1$, we 
have $y^{\epsilon}_{22} = -\epsilon$. Since $x^{\epsilon} + y^{\epsilon}$ must be positive semidefinite, its determinant must be nonnegative so the following inequality must hold
\[
\frac{y^{\epsilon}_{11}(\epsilon^2+\epsilon^3) + 1}{\epsilon^2+\epsilon^3} - \frac{(1+\epsilon y^\epsilon_{12})^2}{\epsilon^2} \geq 0.
\]
Therefore,
\begin{align*}
y_{11}^\epsilon & \geq -\frac{1}{\epsilon^2+\epsilon^3} + \frac{1}{\epsilon^2} + \frac{2y_{12}^\epsilon}{\epsilon} + (y_{12}^{\epsilon})^2\\
& = \frac{1}{\epsilon(1+\epsilon)}  + \frac{2y_{12}^\epsilon}{\epsilon} + (y_{12}^{\epsilon})^2.
\end{align*}
By \eqref{eq:bound3}, $|y_{12}^\epsilon|$ is bounded above by $(\kappa+1)\epsilon$. We then have
\[
y_{11}^\epsilon \geq \frac{1}{\epsilon(1+\epsilon)}  - 2(\kappa+1). 
\]
As $\epsilon$ goes to $0$, $y_{11}^\epsilon$ goes to $+\infty$, which contradicts \eqref{eq:bound3}. We conclude that 
\eqref{eq:bound_cd} cannot possibly hold.
Therefore, although $C$ is amenable, the amenability of its faces must be considered locally. 
\end{example}

\subsection{Basic properties of amenable convex sets}
In this subsection we prove some basic properties of amenable convex sets.
\begin{proposition}[Properties of convex amenable sets]\label{prop:am_int}
	Let $C_1,C_2 \subseteq \ambSpace$ be  convex sets and $\hat \ambSpace$ a finite dimensional Euclidean space.
	\begin{enumerate}[$(i)$]
		\item \label{prop:am_int:1} If $C_1$ and $C_2$ are amenable then $C_1\cap C_2$ is amenable.
		\item \label{prop:am_int:2} If $C_1$ and $C_2$ are amenable then $C_1 \times C_2$ is amenable.
		\item \label{prop:am_int:3} If $A:\ambSpace \to \hat \ambSpace$ is an 
		injective affine map, then $A(C_1)$ is amenable if and only if $C_1$ is amenable.
		\item \label{prop:am_int:4} If $C_1$ is polyhedral, then it is amenable.
		\item \label{prop:am_int:5} $C_1$ is amenable if and only if $C_1 \cap (\lineality C_1)^\perp$ is amenable.		
	\end{enumerate}
\end{proposition}
\begin{proof} 
	\fbox{$(i)$}
	Let $\stdFaceC$ be a face of $C_1\cap C_2$ and let 
	$C \coloneqq C_1 \cap C_2$. 
	By Proposition~\ref{prop:face_intersection}, there are faces $\stdFaceC _1 \face C_1, \stdFaceC_2 \face C_2$ such that \[\stdFaceC = \stdFaceC_1 \cap \stdFaceC_2,\qquad 
	\reInt(\stdFaceC) = \reInt(\stdFaceC_1) \cap \reInt (\stdFaceC_2).\]
	In particular, this implies that 
	\begin{equation}\label{eq:am_int:1}
	\reInt (\stdFaceC _1) \cap \reInt (\stdFaceC _2) \neq \emptyset.
	\end{equation}
	Now, we are ready to show that $\stdFaceC$ is an amenable face of $C$.
	Let $B$ be an arbitrary bounded set. By~\eqref{eq:am_int:1} and Proposition~\ref{prop:blr}, there exists $\hat \kappa$ such that 
	\[
	\dist(x, \stdFaceC) \leq \hat \kappa (\dist(x, \stdFaceC_1)+\dist(x,\stdFaceC_2)), \quad \forall x \in B.
	\]
	Since $\stdFaceC_1$ and $\stdFaceC_2$ are amenable faces of $C_1$ and $C_2$ respectively, there are constants $\kappa _1, \kappa _2$ satisfying \eqref{eq:am_blr}. Therefore, for every $x \in B$ we have
	\[
	\dist(x, \stdFaceC) \leq \hat \kappa (\kappa_1\dist(x, C_1)+\kappa_1\dist(x, \aff \stdFaceC_1) +\kappa_2\dist(x, C_2)+\kappa_2\dist(x, \aff \stdFaceC_2)).
	\]
	Since $C \subseteq C_1\cap C_2$ and $\aff \stdFaceC \subseteq (\aff \stdFaceC _1) \cap (\aff\stdFaceC_2)$, we have
	for $i \in \{1,2\}$ 
	\[
	\dist(x,C_i) \leq \dist(x, C), \quad \dist(x, \aff \stdFaceC _i) \leq \dist(x, \aff \stdFaceC), \quad \forall x \in B.
	\]
	Letting $\kappa \coloneqq 2\hat \kappa \max\{\kappa_1, \kappa_2\}$, we conclude that
	\[
	\dist(x,\stdFaceC) \leq \kappa (\dist(x, C)+\dist(x,\aff \stdFaceC)), \qquad \forall x \in B.
	\]
	By Proposition~\ref{prop:am_eq}, this implies that $\stdFaceC$ is an amenable face of $C$.
	
	\fbox{$(ii)$} We assume  $\ambSpace \times \ambSpace$
	is equipped with a norm such that 
	\[
	\norm{(x,y)} = \norm{x} + \norm{y}, \quad \forall (x,y) \in \ambSpace.
	\]
	Let $\stdFace \face C_1 \times C_2$, then there are $F_1 \face C_1$ and $F_2 \face C_2$ such that $F = F_1 \times F_2$.
	
	Let $B$ be a bounded set in $\ambSpace \times \ambSpace$. 
	We denote by $B_1$ and $B_2$ the projection of $B$ on the first and second coordinate variables, respectively.
	Since $B_1,B_2$ are bounded and $F_1,F_2$ are amenable faces, there 
	are positive constants $\kappa_1,\kappa_2$ satisfying the definition of amenability \eqref{eq:am_def}.
	With that,  let $(x,y) \in \aff F = (\aff F_1) \times (\aff F_2)$ be such that $(x,y) \in B$. We have
	\begin{align*}
  	\dist((x,y),F) & = \dist(x,F_1) + \dist(y,F_2)\\
  	& \leq \max\{\kappa_1,\kappa_2\}(\dist(x,C_1)+\dist(y,C_2))\\
  	&=\max\{\kappa_1,\kappa_2\}(\dist((x,y),C_1\times C_2)),
	\end{align*}
	which completes the proof of item $(\ref{prop:am_int:2})$.
	As a remark, we note that because of the equivalence of norms on finite-dimensional spaces, it does not matter which norm we use in $\ambSpace \times \ambSpace$, except that the constants might change.
	
	\fbox{$(iii)$} Since $\stdMap$ is an injective affine map, there exists some injective linear map $\mathcal{B}$ and $y_0 \in \hat \ambSpace$ such that
	\[
	\stdMap(x) = y_0 + \mathcal{B}(x), \qquad \forall x \in \ambSpace.
	\]
	If $\mathcal{B}$ is the zero operator, we are done because a set with a single point is always amenable. 
%
	So, first suppose that $C_1$ is amenable. We note that $\stdFaceC \face C_1$ if and only if $\stdMap(\stdFaceC) \face \stdMap(C_1)$. Furthermore, we have $\aff \stdMap(\stdFaceC) = \stdMap(\aff \stdFaceC)$.
	
	Let $B \subseteq \hat \ambSpace$ be a bounded set and $\stdMap(\stdFaceC)$ be a face of $\stdMap(C_1)$. 
	Because $\stdMap$ is injective, $\stdMap^{-1}(B)$ is bounded in $\ambSpace$, so there exists $\kappa > 0$ such that 
	\begin{equation}\label{eq:iso1}
	\dist(x, \stdFaceC) \leq \kappa \dist(x, C_1), \quad \forall x \in (\aff \stdFaceC )\cap \stdMap^{-1}(B).
	\end{equation}
	
	Let $\sigma_{\max}$ and $\sigma _{\min}$ denote the 
	maximum and minimum singular values of $\mathcal{B}$, so that 
	\[
	\sigma _{\max} = \max \{ \norm{\mathcal{B}(x)} \mid \norm{x} = 1 \}, \qquad 	\sigma _{\min} = \min\{ \norm{\mathcal{B}(x)} \mid \norm{x} = 1 \},
	\]
	where we also use $\norm{\cdot}$ to denote the norm in $\hat \ambSpace$.
	Because $\mathcal{B}$ is injective and is not the zero operator, we have $\sigma _{\min} > 0$.
	
	Let $\stdMap(x) \in B\cap (\stdMap(\aff \stdFaceC))$, 
	we have
	\begin{align*}
	\dist(\stdMap(x), \stdMap(\stdFaceC)) & = \dist(\mathcal{B}(x), \mathcal{B}(\stdFaceC))\\
	 & \leq \sigma_{\max} \dist(x, \stdFaceC)\\
	 & \leq \sigma_{\max}\kappa \dist(x, C_1)\\
	 & \leq \frac{\sigma_{\max}}{\sigma_{\min}}\kappa \dist(\mathcal{B}({x}), \mathcal{B}(C_1))\\
	  & = \frac{\sigma_{\max}}{\sigma_{\min}}\kappa \dist(\stdMap({x}), \stdMap(C_1)),
	\end{align*}
	where the second inequality follows from \eqref{eq:iso1}. This shows that $\stdMap(C_1)$ is amenable. The converse is analogous, so it is omitted.
	
	\fbox{$(iv)$} First, we note that a closed half-space $H^+$ must be amenable. It only has two faces, $H^+$ itself and the underlying hyperplane which we denote by $H$. Since $H$ is an affine set, we have $\aff H = H$, and the amenability condition \eqref{eq:am_def} is satisfied.
	
	Since any polyhedral set can be expressed as an intersection of finitely many closed half-spaces, it must be amenable by item~\eqrefit{prop:am_int:1}.
	
	\fbox{$(v)$} Suppose that $C_1$ is amenable. 
	Since $\lineality C_1 ^\perp$ is a subspace, by item~$(\ref{prop:am_int:4})$, $\lineality C_1 ^\perp$ is amenable. 
	Then, $C_1 \cap (\lineality C_1) ^\perp $ is amenable by item~$(\ref{prop:am_int:1})$.
	Conversely, suppose that $C_1 \cap (\lineality C_1) ^\perp $ is amenable. Then, $(C_1 \cap (\lineality C_1 ^\perp)) \times \lineality C_1$ is amenable by items~$(\ref{prop:am_int:2})$ and $(\ref{prop:am_int:4})$. Since $C_1$ is isomorphic to 
	to $(C_1 \cap (\lineality C_1) ^\perp) \times \lineality C_1$, $C_1$ is amenable by item~$(\ref{prop:am_int:3})$.
\end{proof}
In what follows, we recall that the \emph{doubly nonnegative cone $\doubly{n}$} is the cone of $n\times n$ real symmetric matrices which are positive semidefinite and have nonnegative entries. Next, a \emph{spectrahedral set} $C$ is defined to be the intersection of 
an affine space $\stdAffine \subseteq \S^n$ with  $\PSDcone{n}$ or anything linearly isomorphic
to $\stdAffine \cap \PSDcone{n}$.
%
We also recall that a closed convex cone $\stdCone$ is said to be \emph{homogeneous} if its group of automorphisms acts transitively in the interior of $\stdCone$.
\begin{corollary}\label{cor:am}
	The following convex sets are amenable.
	\begin{enumerate}[$(i)$]
		\item The doubly nonnegative cone $\doubly{n}$.
		\item Spectrahedral sets.
		\item Homogeneous cones.
	\end{enumerate}
\end{corollary}
\begin{proof}
	\fbox{$(i)$} The doubly nonnegative cone $\doubly{n}$ is the intersection of $\PSDcone{n}$ and the cone of symmetric nonnegative matrices, which are both amenable. Therefore, $\doubly{n}$ is amenable by item~\eqrefit{prop:am_int:1} of Proposition~\ref{prop:am_int}.
	
	\fbox{$(ii)$} Let $\stdAffine$ be an affine space.
	Because the cone of symmetric positive semidefinite matrices $\PSDcone{n}$ is amenable (\cite[Proposition~33]{L19}), an intersection of the format $\stdAffine \cap \PSDcone{n}$ or anything linearly isomorphic to $\stdAffine \cap \PSDcone{n}$ must be amenable by items \eqrefit{prop:am_int:1}, \eqrefit{prop:am_int:3} and \eqrefit{prop:am_int:4} of Proposition~\ref{prop:am_int}.
	
		\fbox{$(iii)$} Chua \cite{CH03} (see also Proposition 1 and  Section 4 of the paper by Faybusovich \cite{FB02}) showed that homogeneous cones are ``slices'' of the positive semidefinite cone. The precise statement is that if $\stdCone$ is a homogeneous cone in $\Re^m$, there exists $n \geq m$ and an injective linear map $M$ such that 
		\[
		M(\reInt \stdCone) = (\reInt\PSDcone{n}) \cap M(\Re^m),
		\]
		see \cite[Corollary~4.3]{CH03}. In particular, 
		we have $M(\stdCone) = \PSDcone{n} \cap M(\Re^m)$, which shows that $\stdCone$ is a spectrahedral set and must be amenable by item $(ii)$.	
\end{proof}
\begin{remark}
The arguments in Corollary~\ref{cor:am} can be used to show that the feasible region $S$ of a conic linear program where the underlying cone is amenable must also be amenable. This follows from item~\eqrefit{prop:am_int:1} of Proposition~\ref{prop:am_int} when $S$ is expressed as the intersection of an affine space and an amenable cone. Next, suppose that $S$ is written as $\{y \mid c - \stdMap y \in \stdCone \}$, where $c$ is a vector and $\stdMap$ is an injective linear map of appropriate dimensions. With that, we have  
$ S = \stdMap^{-1}((c-\stdCone)\cap \matRange \stdMap)$ so that $S$ is amenable if $\stdCone$ is amenable, by items~\eqrefit{prop:am_int:1} and ~\eqrefit{prop:am_int:3} of Proposition~\ref{prop:am_int}.
\end{remark}

Corollary~\ref{cor:am} solves a few of the questions that were outlined in the conclusion of \cite{L19}, in particular whether homogeneous cones are amenable or not. In addition, although error bounds for the doubly nonnegative cone were shown in \cite{L19}, the amenability of $\doubly{n}$ was left open. 

We note that the amenability of $\doubly{n}$ has the following curious consequences.
First, it shows that the completely positive cone $\mathcal{CP}^n$ is amenable for $n \leq 4$, since $\mathcal{CP}^n = \doubly{n}$ for $n \leq 4$.  
However, $\mathcal{CP}^n$ is not amenable for $n \geq 5$ because it is not facially exposed, see \cite{Zh18}. Nevertheless, the fact that $\mathcal{CP}^4$ is amenable gives an explicit example of a cone that is amenable but whose dual cone is not: the dual of 
$\mathcal{CP}^4$ is the cone of $4\times 4$ symmetric copositive matrices which is known to not be facially exposed.

Next, we will discuss some inheritance properties of amenability. In what follows we say that a face 
$\stdFaceC \face C$ is \emph{maximal} if $\stdFaceC \neq C$ and there is no face $\hat \stdFaceC \face C$ satisfying $\hat \stdFaceC \neq \stdFaceC$, $\hat \stdFaceC \neq C$ and 
$\stdFaceC \face \hat \stdFaceC \face C$.
\begin{proposition}[Inheritance and transitivity of amenability]\label{prop:inh}
Let $C$ be a closed convex set. The following items hold:
\begin{enumerate}[$(i)$]
	\item\label{prop:inh:tr} (Transitivity) Let $\hat \stdFaceC$ and $\stdFaceC$ be faces satisfying $\hat \stdFaceC \face \stdFaceC \face C$, where $\stdFaceC$ is an amenable face of $C$. Then, $\hat \stdFaceC$ is an amenable face of $\stdFaceC$ if and only if it is an amenable face of $C$.
	\item (Inheritance) If $C$ is amenable, then 
	every face $\stdFaceC \face C$ is an amenable convex set by itself.
	\item $C$ is amenable if and only if every maximal face $\stdFaceC\face C$ is both an amenable face of $C$ and an amenable convex set by itself.
\end{enumerate}
\end{proposition}
\begin{proof}
\fbox{$(i)$} Suppose that $\hat \stdFaceC$ is an amenable face of 
$C$ and let $B$ be a bounded set. Since 
$\stdFaceC \subseteq C$, we have $\dist(x, C) \leq \dist(x, \stdFaceC)$, for every $x \in B$. In view of Definition~\ref{def:am} and \eqref{eq:am_def}, $\hat \stdFaceC$ must be an amenable face of $\stdFaceC$ as well. Conversely, 
suppose that $\hat \stdFaceC$ is an amenable face of $\stdFaceC$. By assumption, $\stdFaceC$ is an amenable face of $C$, so by Proposition~\ref{prop:am_eq},
 there exists $\kappa _B > 0$ such that 
\begin{equation}\label{eq:inh_1}
\dist(x,\stdFaceC) \leq \kappa _B \max \{\dist(x,\aff \stdFaceC), \dist(x,C) \},\qquad \forall x \in B.
\end{equation}
Similarly, since $\hat \stdFaceC$ is an amenable face of $\stdFaceC$, there exists $\hat \kappa _B > 0$ such that
\begin{equation}\label{eq:inh_2}
\dist(x,\hat \stdFaceC) \leq \hat \kappa _B \max \{\dist(x,\aff \hat \stdFaceC), \dist(x,\stdFaceC) \},\qquad \forall x \in B.
\end{equation}
Combining \eqref{eq:inh_1} and \eqref{eq:inh_2} and 
using the fact that $\dist(x, \aff \stdFaceC) \leq \dist(x, \aff \hat \stdFaceC)$, we conclude that 
$\hat \stdFaceC$ is an amenable face of $C$.

\fbox{$(ii)$} 
 Every face $\stdFaceC \face C$ satisfies
$\stdFaceC \coloneqq C \cap \aff \stdFaceC.$
Therefore, if $C$ is an amenable cone, by item~\eqrefit{prop:am_int:1} of Proposition~\ref{prop:am_int}, $\stdFaceC$ must be an amenable convex set by itself.

\fbox{$(iii)$} If $C$ is amenable, by item $(ii)$, all the maximal faces must be amenable convex sets as well.
Conversely, suppose that $C$ is such that every maximal face is an amenable face and an amenable convex set by itself. Let $\hat \stdFaceC \face C$ be an arbitrary face. Because every proper face is contained in a maximal face, $\hat \stdFaceC$ must be a face of some maximal face $\stdFaceC$.  By assumption, $\stdFaceC$ is both an amenable face of $C$ and a convex amenable set by itself, so $\hat \stdFaceC$ must be an amenable face of $C$ by item~\eqrefit{prop:inh:tr}.
\end{proof}

\begin{remark}[Set operations and notions of exposedness]
Propositions~\ref{prop:am_int} and Proposition~\ref{prop:inh} shows that amenability 
is preserved by quite a few set operations. 
We compare briefly how other notions of exposedness fare in this regard. See Table~\ref{tab:table1} for a summary.
\begin{itemize}
	\item Facial exposedness of convex sets is also preserved by finite intersections, direct products, injective linear images. Also, polyhedral sets must be facially exposed, which is a consequence of item \eqrefit{prop:am_int:4} of Proposition~\ref{prop:am_int} and Proposition~\ref{prop:imp} (see also \cite[Corollary~2]{TM76}). 
	It is well-known, however, that facial exposedness does not satisfy transitivity. 
	That is, it can be the case that $\hat \stdFaceC$ is a facially exposed face of $\stdFaceC$, $\stdFaceC$ is a facially exposed face of $C$ but $\hat \stdFaceC$ is not a facially exposed face of $C$. Homogeneous cones are facially exposed \cite{TruongTuncel}.
	
	\item Niceness is only defined for cones but is also preserved by finite intersections (see \cite[Proposition~5]{pataki}), direct products, injective linear images. Furthermore, niceness is transitive and inherited by the faces of nice cones. The former follows directly from the definition of niceness. The latter follows from the fact that a face $\stdFace \face \stdCone$ satisfies $\stdFace = \stdCone \cap \lspan \stdFace$ and  the intersection of nice cones is nice.
	Homogeneous cones are nice, 
	see \cite[Proposition~4]{ChuaTuncel} and \cite{CH03}. This also follows from  Corollary~\ref{cor:am} and Proposition~\ref{prop:imp}.
	\item Projectional exposedness is also only defined for cones and it is preserved by direct products and injective linear images. Polyhedral cones must be projectionally exposed, see \cite{BLP87} and \cite[Corollary~3.4]{ST90}. 
	Symmetric cones are known to satisfy a stronger form of projectional exposedness where the projections can be chosen to be orthogonal, see \cite[Proposition~33]{L19}, but it is unknown whether homogeneous cones are projectionally exposed in general.
	Notably, it is not known whether projectional exposedness is preserved by intersections.
	Nevertheless, projectionally exposedness is transitive and is inherited by the faces of projectionally exposed cones, as shown in Lemmas~2.2 and 2.3 of \cite{ST90}.
\end{itemize} 	
\end{remark}

\begin{table}[ht]
\def\arraystretch{1.5}
\begin{tabularx}{\textwidth}{ll|Y|Y|Y|Y}
                                                                                 &                       
                                                                                  
& Facially \mbox{Exposed} & Nice & Amenable & Projectionally Exposed \\ \hline
\multicolumn{2}{l|}{Defined for convex sets}                                                      
& \centering \cmark    & \centering \xmark                               & \cmark    & \xmark                     \\ 
\hline
\multicolumn{1}{l|}{ \multirow{3}{*}{\makecell[l]{Preserved \\  under}}} & \makecell[l]{finite\\ \mbox{intersections}}   
&  \cmark    & \cmark    &  \cmark     & \qmark               \\ 
\cline{2-6} 
\multicolumn{1}{l|}{}                                                            & \makecell[l]{direct\\ product}         
&  \cmark    &  \cmark    &  \cmark    &\cmark                  \\ 
\cline{2-6} 
\multicolumn{1}{l|}{}                                                            & \makecell[l]{injective \\ \mbox{linear image}} 
&  \cmark    &  \cmark    &  \cmark    & \cmark                  \\ 
\hline
\multicolumn{2}{l|}{Face transitive}                                                                         
& \xmark    &  \cmark    & \cmark       & \cmark                 \\ 
\hline
\multicolumn{2}{l|}{Symmetric cones}                                                       
&  \cmark   &  \cmark   &  \cmark    & \cmark                     \\ 
\hline
\multicolumn{2}{l|}{Homogeneous cones}                                                       
&  \cmark   &  \cmark   &  \cmark    & \qmark                     \\ 
\hline
\end{tabularx}

\centering
\caption{Relations between different notions}
\label{tab:table1}

\end{table}

\section{Slices of amenable cones}\label{sec:slices}
Let $\stdCone$ be a pointed closed convex cone. Then, it can be shown 
that $\stdCone$ is generated by a compact ``slice'' as follows.
Let $e \in \reInt \stdCone^*$ and define 
\[
C \coloneqq \{x\in \stdCone \mid \inProd{x}{e} = 1 \}.
\]
With that, $C$ is compact and $\stdCone$ is the cone generated 
by $C$. Naturally, many properties of $C$ are transferred to $\stdCone$ and vice-versa.

In this subsection, we take a look at how amenability is transferred from 
$C$ to $\stdCone$. We start with the following observation.
\begin{proposition}[Polyhedral cuts preserve amenability]\label{prop:slices}
Let $\stdCone$ be an amenable closed convex cone and let $P$ be a polyhedral set. Then $\stdCone \cap P$ is an amenable convex set. 
In particular, if $\stdCone$ is pointed, then $\stdCone$ is generated by a 
compact amenable slice.
\end{proposition}
\begin{proof}
Since $P$ is polyhedral, 
$\stdCone \cap P$ is amenable by items~\eqrefit{prop:am_int:1} and \eqrefit{prop:am_int:4} of Proposition~\ref{prop:am_int}.

For the second part, let $e \in \reInt \stdCone^*$ and define $C \coloneqq \{x\in \stdCone \mid \inProd{x}{e} = 1 $. As remarked previously, $C$ is compact and is the intersection of 
$\stdCone$ and the hyperplane $\{x \in \ambSpace \mid \inProd{x}{e} = 1  \}$. Therefore, $C$ is amenable and $\stdCone = \cone C$.
\end{proof}
Next, we take a look at the converse of Proposition~\ref{prop:slices} and check whether the cone generated by an amenable compact convex set is amenable. This is a harder question and requires some careful analysis. Before we state and prove the result in Theorem~\ref{thm:slice}, we need a few preparatory results.

\begin{proposition}\label{prop:technicalslice}  
	Let $\stdCone = \cone C$, where $C\subseteq \ambSpace$ is a compact convex set contained in the hyperplane 
	\[
	H = \{x \in \ambSpace \mid \inProd{e}{x} = 1\},
	\]
	where $e\in \ambSpace$ is nonzero. Then for every $x\in H\setminus (-\stdCone^*)$  
	\[
	\dist (x,C)\leq  \|e\| r\dist (x,\stdCone),
	\]
	where $r = \max_{u\in C}\|u\|$.
\end{proposition}
\begin{proof} Let $u\in C$ and let $v$ be such that $\langle v, e\rangle = 0$ and $\norm{v} \neq 0$. Then 
	\begin{equation}\label{eq:boundprod}
	\left|\left\langle \frac{u}{\|u\|}, \frac{v}{\|v\|}\right\rangle\right| =\frac{1}{\|u\|} \left| \left\langle u - \frac{1}{\|e\|^2}e, \frac{v}{\|v\|}\right\rangle\right| \leq \frac{\|u - \frac{1}{\|e\|^2} e\|}{\|u\|}. 
	\end{equation}
	Observe that 
	\begin{equation}\label{eq:boundtechnical}
	\frac{\|u - \frac{1}{\|e\|^2} e\|^2}{\|u\|^2} = \frac{\|u\|^2 - \frac{1}{\|e\|^2} }{\|u\|^2} = 1 - \frac{1}{\|u\|^2 \|e\|^2}\leq 1 - \frac{1}{r^2 \|e\|^2}.
	\end{equation}	
	Hence from \eqref{eq:boundprod} and \eqref{eq:boundtechnical} we have for any $u\in C$ and any $v$ such that $\langle v, e\rangle = 0$ that
	\begin{equation}\label{eq:boundtechnical2}
	|\langle u,v\rangle| \leq \|v\|\|u\| \sqrt{1-\frac{1}{\|e\|^2 r^2}}.
	\end{equation}
	Since for every $w\in \stdCone$ we have $w = \lambda u$ where $u\in C$ and $\lambda\geq 0$,  from \eqref{eq:boundtechnical2} we obtain 
	\begin{equation}\label{eq:prodbound}
	|\langle w,v\rangle| \leq \|v\|\|w\| \sqrt{1-\frac{1}{\|e\|^2 r^2}} \quad \forall w\in \stdCone, \; \forall v\, \text{ s.t. } \langle v, e\rangle  = 0.
	\end{equation}
	
	Now let $x\in H\setminus (-\stdCone^*)$, and let $y$ be the projection of $x$ onto $\stdCone$. Since $y\in \stdCone = \cone C$, there is $z\in C$ and $\lambda\geq 0$ such that $y = \lambda z$. Moreover, since $x\notin - \stdCone^*$ we know that $y\neq 0$, hence $\lambda \neq 0$.  Since $\stdCone$ is a cone, we deduce that $\langle x-y,y\rangle  = \langle x-y, z\rangle= 0$, and hence
	\begin{equation}\label{eq:triplesquare}
	\|x-z\|^2 = \|x-y\|^2 + 2 \langle x-y, y-z\rangle + \|y - z\|^2 = \|x-y\|^2 + \|y-z\|^2.
	\end{equation}
	Furthermore,
	\begin{equation}\label{eq:nmbound}
	\|y-z\|^2 = \langle y - z, y - z\rangle = \langle y - z, y - x\rangle + \langle y - z, x - z\rangle =  \langle y-z, x - z\rangle = |\langle y-z, x - z\rangle|.
	\end{equation}
	Since $x,z\in H$, we have $\langle x-z,e\rangle = 0$, and also either $y-z$  is in $\stdCone$ (if $\lambda \geq 1$) or $z-y$  is in $\stdCone$ (if $\lambda \leq 1$). Hence from \eqref{eq:prodbound} and \eqref{eq:nmbound}
	\[
	\|y-z\|^2 = |\langle y-z, x - z\rangle| \leq \|y-z\|\|x-z\| \sqrt{1-\frac{1}{\|e\|^2 r^2}}.
	\]
	In the case when $\lambda\neq 1$ (and hence $\|y-z\|\neq 0$) we can cancel $\|y-z\|$. Taking squares on both sides and using \eqref{eq:triplesquare} we have
	\[
	\|x-z\|^2 \leq \|x-y\|^2 + \|x-z\|^2 \left(1-\frac{1}{\|e\|^2 r^2}\right),
	\]
	hence
	\[
	\dist(x,C)^2 \leq \|x-z\|^2 \leq \|e\|^2 r^2 \|x-y\|^2 = \|e\|^2 r^2 \dist(x,\stdCone)^2.
	\]
	When $\lambda = 1$, we have $y=z$, and hence 
	\[
	\dist (x, C) \leq \|x-z\| = \|x-y\| = \dist (x, \stdCone)\leq  \|e\| r \dist (x, \stdCone),
	\]
	where the last inequality follows from observing that $\|e\| \geq \frac{\langle u, e\rangle}{\|u\|} = \frac{1}{\|u\|}\geq \frac{1}{r}$.
\end{proof}

Our next result is a geometrically intuitive claim on the existence of a universal upper bound on the angle between a closed convex cone and any vector in the linear span of this cone, given that this cone is not one-dimensional (see Fig.~\ref{fig:angles}).
\begin{figure}[ht]
	\centering\includegraphics[width=0.45\textwidth]{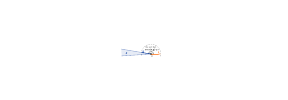}\quad 
	\caption{The intuition behind Proposition~\ref{prop:alpha0}: the angle between a `thick' cone and any vector in its linear span is (uniformly) strictly less than $\pi$.}
	\label{fig:angles}
\end{figure}

\begin{proposition}\label{prop:alpha0} Suppose that $\stdCone\subseteq \ambSpace$ is a closed convex cone. If $\dim \stdCone >1$, then 
	\[
	\alpha:= \inf_{\substack{x\in \lspan \stdCone,\\ \|x\|=1}}\sup_{\substack{y\in \stdCone,\\ \|y\|=1}}\langle x,y\rangle >-1.
	\]
\end{proposition}
\begin{proof} Suppose that the statement is not true. Then there exists a cone $\stdCone$ such that $\dim \stdCone >1$ and a  sequence $\{x_n\}$ such that $x_n\in \lspan \stdCone$, $\|x_n\|=1$ for all $n$, and
	\[
	\lim_{n\to \infty} \sup_{y\in \stdCone, \|y\|=1}\langle x_n,y\rangle = -1.
	\]
	Since $\stdCone$ is closed, for every $x_n$ there is $y_n\in \stdCone$, $\|y_n\| = 1$ such that
	\[
	\sup_{y\in \stdCone, \|y\|=1}\langle x_n,y\rangle = \langle y_n,x_n\rangle.
	\]
	Moreover, by compactness we can assume that  $x_n\to \bar x \in \lspan \stdCone$, $y_n \to \bar y \in \stdCone$, $\|\bar x\|=1$, $\|\bar y\| = 1$, and $\langle\bar x , \bar y \rangle = -1$, equivalently $\bar y = - \bar x \in \stdCone$. Since $\dim \stdCone>1$, there exists $z\in \stdCone$, $\|z\|=1$, such that $z$ is linearly independent with $\bar x,\bar y$.  We then have $\langle \bar x,z\rangle  > -1$, and since $z\in \stdCone$, $\|z\|=1$,  
	\[
	-1 = \lim_{n\to \infty} \sup_{y\in \stdCone, \|y\|=1}\langle x_n,y\rangle \geq \lim_{n\to \infty} \langle x_n,z\rangle = \langle \bar x , z\rangle, 
	\]
	hence $\langle \bar x,z\rangle \leq -1$, a contradiction. 
\end{proof}

\begin{proposition}\label{prop:technicalcone} Let $\stdCone \subseteq \ambSpace$ be a closed convex pointed cone. Then for any face $\stdFace \face \stdCone$ with $\dim \stdFace>1$ there exists  $\beta>0$ such that for any $x\in \lspan \stdFace$ 
and any
	\begin{equation}\label{eq:xofy}
	y \in \Argmax_{\substack{u\in \stdFace, \\\|u\|=1} } \langle x,u\rangle 
	\end{equation}
	we have 
	\begin{equation}\label{eq:mainbounds}
	\dist (x+ t y, \stdCone)\leq \dist (x,\stdCone), \quad \dist(x,\stdFace) \leq \beta \dist (x+ t y, \stdFace),\quad  \forall t\geq 0.
	\end{equation}
\end{proposition}
\begin{proof} Assume that $\stdFace \face \stdCone$ is a face such that $\dim \stdFace>1$. From Proposition~\ref{prop:alpha0} we know that there exists a constant $\alpha>-1$ such that 
	\begin{equation}\label{eq:uniformq}
	\sup_{u\in \stdFace, \|u\|=1}\langle x,u\rangle \geq \alpha\quad  \forall  x\in \lspan \stdFace,\, \|x\|=1.
	\end{equation}
	We recall that for every cone $\stdFace$ we have $\reInt \stdFace \cap \reInt \stdFace^* \neq \emptyset$\footnote{If $\reInt \stdFace \cap \reInt \stdFace^* = \emptyset$, there is a hyperplane passing through the origin that properly separates $\stdFace$ and $\stdFace^*$. Letting $z$ denote the (nonzero) normal of this hyperplane, we may assume that $\inProd{z}{x} \leq \inProd{z}{y}$ for every $x \in \stdFace^*, y \in \stdFace$. Therefore, $z \in \stdFace^* \cap (-\stdFace) = \{0\}$, which is a contradiction.}.
	Therefore, there exists at least one nonzero element of $\lspan \stdFace$ that belongs to $-\stdFace^*$. In view of \eqref{eq:uniformq},  it must be the case 
	that $\alpha \in (-1,0]$. With that in mind, we let $\beta: = \frac{1}{\sqrt{1-\alpha^2}}\in [1,+\infty)$.
	
Let $x \in \lspan \stdFace$ and suppose that $y$ satisfies \eqref{eq:xofy}. Let $z$ be the projection of $x$ onto $\stdCone$. Observe that since $\stdCone$ is a cone, $z\in \stdCone$ and $t y \in \stdFace\subseteq \stdCone$,  we have $z+ t y\in \stdCone$ and hence we have the first inequality of \eqref{eq:mainbounds}
	\[
	\dist(x+ ty,\stdCone)\leq \|(x+ ty)-(z+ ty)\|  =\|x-z\| = \dist(x, \stdCone).
	\]

	To show the second inequality, first consider the case when $x\notin -\stdFace^*$. Then the projection $z$ of $x$ onto $\stdFace$ is not zero, moreover, the $\Argmax$ function \eqref{eq:xofy} is single-valued at $x$, and this unique value is $y=z/\|z\|$. Indeed, 
	by the properties of the Moreau decomposition, we have 
	\begin{equation}\label{eq:md}
	\inProd{x-z}{z} = 0 \quad \text{and} \quad x-z \in -\stdFace^*.
	\end{equation}
	For $u\in \stdFace$ such that $\|u\|=1$ this yields $	\langle x,u\rangle \leq \langle z,u\rangle \leq \|z\|\|u\| = \|z\|$. 
	We conclude that 
	\[
	\langle x, u\rangle \leq \|z\| =  \frac{\inProd{z}{z} }{\|z\|} =  \frac{\inProd{x}{z} }{\|z\|} \quad \forall u \in \stdFace, \; \|u\|=1,
	\]
	hence,  $y=z/\|z\|$ satisfies \eqref{eq:xofy}. To show that such $y$ is unique, assume that we have another $y'\in \stdFace $, $\|y'\|=1$ such that $\langle x,y\rangle = \langle x,y'\rangle$. Since $\stdFace$ is pointed, the vectors $y$ and $y'$ are noncollinear; moreover, $y+y'\in \stdFace$. Hence, we have  
	\[
	\left\langle x, \frac{y+y'}{\|y+y'\|} \right\rangle > \frac{\inProd{x}{y}+\inProd{x}{y'}}{\|y\|+\|y'\|} = \inProd{x}{y},
	\]
	contradicting the earlier established fact that $y$ maximises the product $\inProd{x}{u}$ over $u\in \stdFace$, $\|u\|=1$.

	From \eqref{eq:md} and the fact that $y \in \stdFace$, we have
	\begin{equation}\label{eq:md2}
	x+ ty =\underbrace{x-z}_{\in -\stdFace^*}  + \underbrace{z+ty}_{\in \stdFace}. 
	\end{equation}
	From \eqref{eq:md} and $y = z/\norm{z}$, we have $\inProd{x-z}{z+ty} = 0$. 
	By the properties of the Moreau decomposition and since \eqref{eq:md2} holds, $z+ty$ must be the projection of 
	$x+ty$ onto $\stdFace$. We conclude that if $x \not \in -\stdFace^*$ and $t \geq 0$ 
	we have
	\begin{equation}\label{eq:f_first}
	\dist (x+ ty,\stdFace) = \|x+ ty - (z+ty)\| = \|x-z\| = \dist (x,\stdFace).
	\end{equation}
	%
	In the remaining case when $x\in - \stdFace^*$, the projection of $x$ onto $\stdFace$ is zero,
	and 
	\begin{equation}\label{eq:xf}
	\dist (x,\stdFace) = \|x\|.
	\end{equation}
	For every $y$ as in \eqref{eq:xofy} and every $u \in \stdFace$, we have 
	\begin{equation}\label{eq:xyxu}
	\inProd{x}{y}\norm{u} \geq \inProd{x}{u}.
	\end{equation}
	Using \eqref{eq:xyxu} and the fact that $\langle x,y \rangle \leq 0$ (since $x\in - \stdFace^*$),  for all $0\leq  t\leq -\langle x,y\rangle$ and all $u\in \stdFace$ such that $\langle y,u\rangle \geq 0$
	\[
	\langle x + t y , u\rangle = \langle x,u\rangle + t \langle y, u\rangle \leq \langle x,y \rangle \|u\| - \langle x, y \rangle \langle y,u\rangle = \langle x, y \rangle (\|u\|- \langle y, u\rangle) \leq 0.
	\] 
	On the the other hand, for all $u\in \stdFace$ such that $\langle u, y\rangle <0$
	\[
	\langle x + t y , u\rangle = \langle x,u\rangle + t \langle y, u\rangle \leq \langle x,y \rangle \|u\|  \leq 0.
	\] 
	hence, $0$ is the projection of $x+ty$ onto $\stdFace$. 
	
	For $t> - \langle x,y\rangle $ let $p_t =\langle x,y \rangle y + t y$.
	We will show that $p_t$ is the projection of $x + ty$ onto $\stdFace$.
	We have\[
	x+ ty = (x - \inProd{x}{y}y) + p_t = (x - \inProd{x}{y}y) + \underbrace{\inProd{x}{y}y+ty}_{\in \stdFace}.
	\]
	A computation using the fact that $\norm{y} = 1$ shows 
	that $\inProd{x - \inProd{x}{y}y}{p_t} = 0$. Therefore, by the properties of the Moreau decomposition, in order to show that $p_t$ is the desired projection, it suffices to check that 
	$x - \inProd{x}{y}y \in -\stdFace^*$. So let $u \in \stdFace$. In view of \eqref{eq:xyxu} and $\inProd{x}{y} \leq 0$, we have
	\begin{align*}
	\inProd{x - \inProd{x}{y}y}{u} &\leq \norm{u}\inProd{x}{y} - \inProd{x}{y}\inProd{y}{u}\\
	&\leq \inProd{x}{y}(\norm{u}- \inProd{y}{u}) \\
	&\leq 0,
	\end{align*}
	hence, $p_t$ is indeed the projection of $x+ t y $ onto $\stdFace$. 
	
	So now we know that if $0 \leq t \leq -\inProd{x}{y}$ then $0$ is the projection of $x+ty$ onto $\stdFace$. Otherwise, if $t > -\inProd{x}{y}$, then $p_t$ is the projection of $x+ty$ onto $\stdFace$.
	Then, from \eqref{eq:uniformq} for every $y$ as in \eqref{eq:xofy} we have $\langle x, y \rangle \geq \alpha \|x\|$. With that and recalling \eqref{eq:xf}, whenever $0 \leq t \leq - \inProd{x}{y} $ we have
	\begin{align*}
	\dist(x+ty, \stdFace) &= \norm{x+ty} \\
	& \geq \min_{0\leq t\leq - \langle x,y\rangle }\|x + t y \|^2\\
	& = \|x\|^2 - \langle x,y\rangle^2\\
	& \geq \|x\|^2 (1- \alpha^2) \\
	& = \dist(x, \stdFace)^2(1- \alpha^2).
	\end{align*}
	On the other hand, if $t > -\inProd{x}{y}$, we have
	\begin{multline*}
		\dist(x+ty, \stdFace)^2=\|x + t y - p_t \|^ 2 = \|x- \langle x,y \rangle y \|^2 =  \|x\|^2 - \langle x,y\rangle^2 \geq \\ \|x\|^2 (1- \alpha^2) = \dist(x,\stdFace)^2 (1- \alpha^2).
	\end{multline*}
	In combination with \eqref{eq:f_first}, we deduce that for all $ x\in \lspan \stdFace$ and all $t\geq 0$
	\[
	\beta\dist (x+ t y , \stdFace) \geq  \dist (x, \stdFace),
	\]
	where $\beta = \max\{1,\sqrt{1/(1-\alpha^2)}\}$.	
\end{proof}

\begin{theorem}[From compact amenable slices to amenable cones]\label{thm:slice}Let $C $ be a compact convex set contained in the hyperplane 
	\[
	H = \{x\, |\, \langle e, x\rangle = 1\},
	\]
	where $e$ is some nonzero vector in $\ambSpace$. If  $C$ is amenable, then its conic hull $\stdCone = \cone C$ is also amenable.
\end{theorem}
\begin{proof}
	Let $\stdFace$ be a face of $\stdCone$. Our goal is to show that there exists a constant $\gamma$ such that 
	\begin{equation}\label{eq:conebound}
	\dist(x,\stdFace)\leq \gamma \dist(x,\stdCone) \quad \forall x\in \lspan \stdFace.
	\end{equation}

	\fbox{Case 1: $\dim \stdFace = 0$} The statement is trivial for $\stdFace = \{0\}$, since in this case $\lspan \stdFace = \stdFace$.

	\fbox{Case 2: $\dim \stdFace = 1$} If $\stdFace$ is one-dimensional and $\stdCone$ is pointed, we have $\stdFace = \cone \{z\}$ for some $z\in  \ambSpace$, $\|z\|=1$, and  $-z\notin \stdCone $. Hence  
	\[
	\gamma : = \dist (-z,\stdCone)>0.
	\]
	In this case for any $x\in \lspan \stdFace$ there is some $\lambda \in \Re$ such that $x = \lambda  z$. Whenever $\lambda\geq 0$, we have $\lambda x \in F \in \stdCone$, and there is nothing to prove. If $\lambda<0$, then  
	\[
	\dist(x,\stdFace) = \|x\| = |\lambda|, \quad \dist(x,\stdCone)  = |\lambda|\dist (-z,\stdCone) = |\lambda | \gamma.
	\]
	Hence we have $\dist(x,\stdFace) \leq \gamma \dist (x,\stdCone)$ for all $x\in \lspan \stdFace$.

	\fbox{Case 2: $\dim \stdFace \geq 2$} It remains to consider the case when $\dim F\geq 2$. There exists a face $E$ of $C$ such that 
	\[
	\stdFace = \cone E, \quad \aff E = \lspan \stdFace\cap H.
	\]
	Since $C$ is amenable and compact, there exists $\kappa>0$ such that
	\begin{equation}\label{eq:004}
	\dist(x,E)\leq \kappa \dist(x,C) \quad \forall x\in \aff E.
	\end{equation}

	By Proposition~\ref{prop:technicalcone} there exists $\beta>0$ such that for any $y$ defined by \eqref{eq:xofy} we have 
	\begin{equation}\label{eq:alpha}
	\dist (x+ t y, \stdCone)\leq \dist (x,\stdCone), \quad \dist(x,\stdFace) \leq \beta \dist (x+ t y, \stdFace),\quad  \forall t\geq 0.
	\end{equation}
	
	Furthermore, since $C$ is compact, there is an $r>0$ be such that 
	\begin{equation}\label{eq:rbound}
	\|u\|\leq r \quad \forall u \in C.
	\end{equation}

	Fix $x\in \lspan \stdFace$. Choose any $y$ satisfying \eqref{eq:xofy}. Since $y\in \stdCone$ and $\|y\|=1$, we have $y = \lambda v$ for some $v\in C$ and $\lambda > 0$.  Hence $\langle e, y\rangle = \lambda \langle e, v\rangle = \lambda > 0 $. Choose any $t$ such that
	\[
	t  > \max\left\{0,-\frac{\langle e,x \rangle }{\langle e,y \rangle }, - \langle x,y\rangle\right\},
	\]
	then for $\bar x = x + t y$ we have 
	$
	\langle e, \bar x\rangle> 0
	$ and  
	\[
	\inProd{\bar x}{y} = \langle x+ t y, y\rangle = \langle x,y\rangle + t \|y\|^2>0,
	\]
	ensuring that $\bar x\notin - \stdCone^*$.
	
	Let $x' = \frac{1}{\langle e,\bar x\rangle} \bar x$. Then $\langle  x',e\rangle =1$, so $x'\in H$, and also $x'\notin -\stdCone^*$, since $\bar x \notin -\stdCone^*$ and $\langle e,\bar x\rangle>0$. We have
	\begin{equation}\label{eq:scale}
	\dist(\bar x, \stdFace)= \langle e,\bar x\rangle\dist(x',\stdFace),\quad  \dist (\bar x, \stdCone) =  \langle e,\bar x\rangle\dist(x',\stdCone).
	\end{equation}
	
	Since $x'\in H\setminus -\stdCone^*$, we can applying Proposition~\ref{prop:technicalslice} to obtain
	\begin{equation}\label{eq:setsineq}
	\dist(x',C)\leq \|e\|r \dist(x',\stdCone),
	\end{equation}
	where $r$ comes from \eqref{eq:rbound}.
	
 From $E\subseteq \stdFace$ we obtain
	\begin{equation}\label{eq:facesineq}
	\dist(x',\stdFace)\leq \dist (x',E).
	\end{equation}

	We have collecting \eqref{eq:alpha}, \eqref{eq:scale} and \eqref{eq:facesineq}
	\begin{equation}\label{eq:002}
	\dist(x,\stdFace)\leq \beta \dist (\bar x, \stdFace) = \beta \langle e,\bar x\rangle\dist(x',\stdFace) \leq \beta \langle e,\bar x\rangle \dist (x',E).
	\end{equation}
	Likewise, from \eqref{eq:alpha}, \eqref{eq:scale} and \eqref{eq:setsineq}
	\begin{equation}\label{eq:003}
	\dist (x,\stdCone) \geq \dist (\bar x, \stdCone) =\langle e,\bar x\rangle\dist(x',\stdCone) \geq \frac{\langle e,\bar x\rangle}{\|e\|r} \dist(x',C). 
	\end{equation}
	Observing that $x'\in \aff E = H \cap \lspan \stdFace$ and combining  \eqref{eq:004}, \eqref{eq:002} and \eqref{eq:003}, we have 
	\[
	\dist(x,\stdFace)\leq  \beta \langle e,\bar x\rangle \dist (x',E)\leq \kappa \beta \langle e,\bar x\rangle \dist (x',C) \leq \kappa  \beta r \|e\|  \dist (x,\stdCone).
	\]
	We conclude that \eqref{eq:conebound} is satisfied with $\gamma =  \beta \kappa  r \|e\|$.
\end{proof}

\section{A nice cone that is not amenable}\label{sec:nicenotamen}

In this section, we produce an explicit example of a closed convex cone in the four-dimensional Euclidean space that is nice but not amenable.

Let $\alpha:[0,2\pi]\rightarrow \RR^3$, $\beta:[0,2\pi]\rightarrow \RR^3$
and $\gamma:[0,\pi]\rightarrow \RR^3$ be defined by
\begin{equation}
\begin{aligned}
\alpha(t) & = (\cos t ,\sin t ,1),\\
\beta(t) & = (\cos t ,\sin t ,-1),\\
\gamma(t) & = \left(2\cos(2t)-1,2\sin(2t),\frac{9}{8} \cos t - \frac{1}{8}\cos(3t)\right).
\end{aligned}\label{eq:curves}
\end{equation}

Throughout this section we use the notation $\alpha$, $\beta$ and $\gamma$ to denote the maps defined in \eqref{eq:curves} and also to denote the sets of points $\alpha([0,2\pi])$, $\beta([0,2\pi])$ and $\gamma([0,\pi])$. The intended meaning should be clear from the context.

We let $\stdCone = \cone (C\times \{1\})$, with $C: = \conv (\alpha \cup \beta \cup \gamma)$. 
The set $C$ is shown in Fig.~\ref{fig:C}.
\begin{figure}[ht]
	\centering\includegraphics[width=0.45\textwidth]{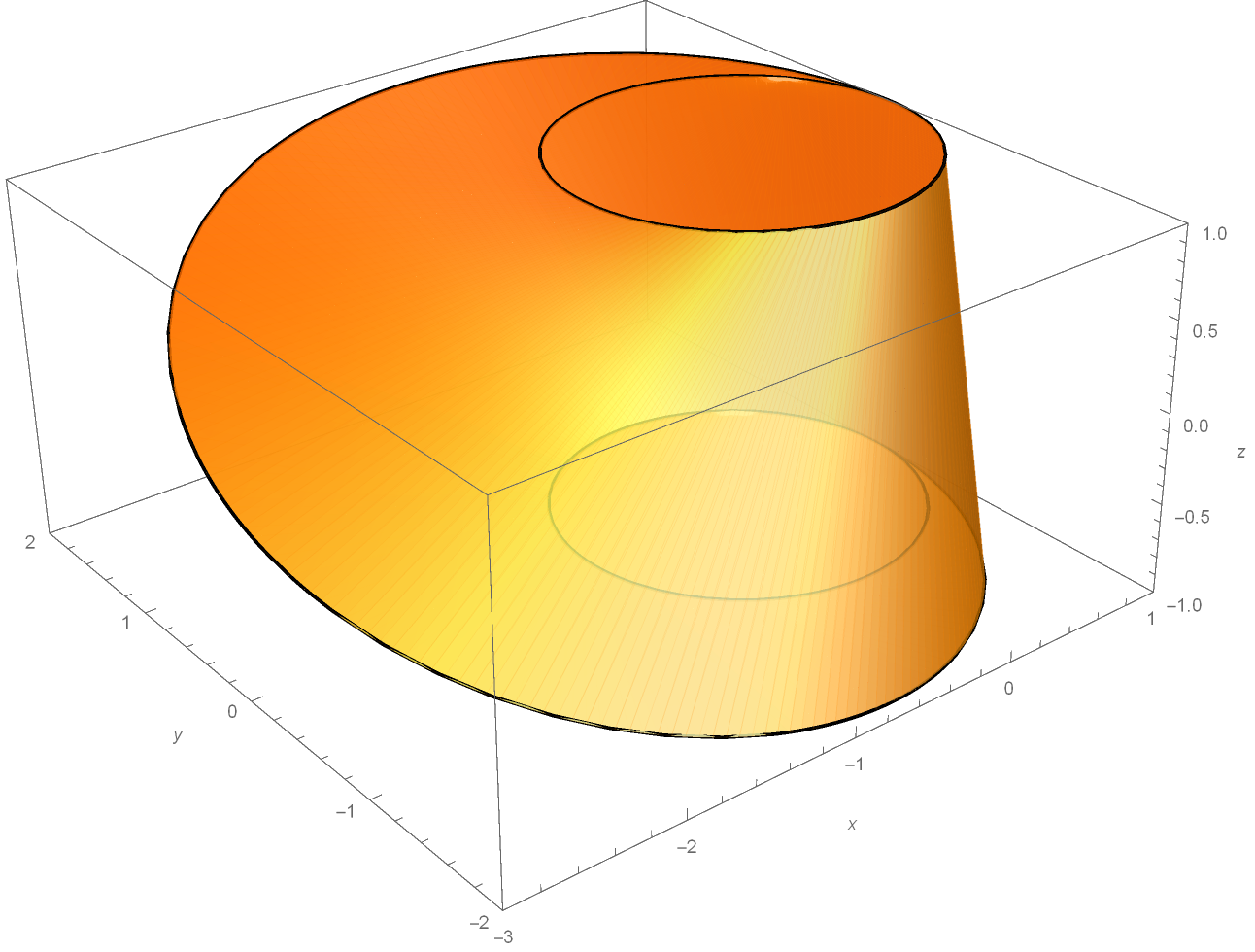}\quad 
	\caption{The set $C$ is the convex hull of the three curves shown in black colour.}
	\label{fig:C}
\end{figure} 

In the next subsection we will prove the following two propositions.

\begin{proposition}\label{prop:fdc} The cone $\stdCone = \cone(C\times\{1\})$ is nice.
\end{proposition}

\begin{proposition}\label{prop:notamen} The cone $\stdCone$ is not amenable.
\end{proposition}
With that, we have the following theorem.
\begin{theorem}\label{thm:fdcnotam} There exists a nice cone $\stdCone\subseteq \R^4$ that is not amenable.
\end{theorem}
\begin{proof}[Proof of Theorem~\ref{thm:fdcnotam}] Follows directly from Propositions~\ref{prop:fdc} and~\ref{prop:notamen}.
\end{proof}

\subsection{Niceness}
\begin{proposition}\label{prop:extreme} Extreme points of $C$ are precisely $\alpha \cup \beta\cup \gamma$. Every extreme point is  exposed. 
\end{proposition}
\begin{proof} From the definition of $C$ we have $\ext C \subseteq \alpha\cup\beta \cup \gamma$. Also note that $\gamma(0)=\alpha (0) = \alpha(2 \pi)$ and $\gamma(2\pi) = \beta(0)=\beta(2 \pi)$. We will first show that for every $t\in (0,2\pi)$ the point $\gamma(t)$ is exposed, and then that each point on $\alpha$ and $\beta$ is also exposed. 
	
	Since the projection of $\gamma$ onto the $xy$-plane is the circle of radius 2 centred at $(-1,0)$, every line that exposes points on this circle as faces of the relevant disk lifts to a plane that likewise exposes individual points $\gamma(t)$ for $t\in (0,2\pi)$. 
	
	The points $\alpha(0)=\alpha(2 \pi )$ and $\beta(0) = \beta(2 \pi)$ are exposed by any plane that exposes them as faces of the cylinder $S = \{(x,y,z)\,|\, x^2+(y-1)^2\leq 4, -1\leq z\leq 1 \}$ that includes $C$ as its subset.
	
	For the remaining points on $\alpha$ and $\beta$ we first observe that due to symmetry it is sufficient to show that $\alpha(t)$ is exposed for all $t\in (0,2\pi)$ (indeed, observe that the linear isometry $T(x,y,z) = (x,-y,-z)$ maps the set $C$ onto itself, swapping the bases: we have  $\beta(2\pi-t) = T(\alpha(t))$ for $t\in [0,2\pi]$ and $\gamma(\pi -t) = T(\gamma(t))$ for $t\in [0,\pi]$). 	The exposing normals have the expression $	p(t) = (\cos t, \sin t, u(t))$,
	where $u(t)$ is chosen in such a way that all points of $\alpha,\beta, \gamma$ except for $\alpha(t)$ lie in the negative half-space defined by the plane via $\alpha(t)$ with positive normal $p(t)$. Explicitly, the following relations must be satisfied,
	\begin{equation}\label{eq:expalpha}
	\langle p(t), \alpha(s)\rangle < \langle p(t), \alpha(t)\rangle  \quad \forall s\in [0,2 \pi ]\setminus \{t\}, \quad \forall t\in (0,2\pi);
	\end{equation}
	\begin{equation}\label{eq:expbeta}
	\langle p(t), \beta(s)\rangle <\langle  p(t), \alpha(t)\rangle  \quad \forall s\in [0,2 \pi ],\quad \forall t\in (0,2 \pi);
	\end{equation}
	\begin{equation}\label{eq:expgamma}
	\langle p(t), \gamma(s)\rangle < \langle p(t), \alpha(t)\rangle  \quad \forall s\in [0,2 \pi ], \quad \forall t\in (0,2 \pi).
	\end{equation}
	Observe that \eqref{eq:expalpha} can be explicitly written as 
	\[
	\cos t \cos s + \sin t \sin s + u(t) < 1+ u(t) \quad \forall s \in [0,2\pi]\setminus \{t\}, \quad \forall t\in (0,2 \pi),
	\]
	or equivalently $\cos(t-s)<1$. Since $s-t\in (-2\pi,0)\cup (0,2\pi)$, this is satisfied for any choice of the function $u$.
	
	From \eqref{eq:expbeta} we have the requirement
	\[
	u(t) >\frac{\cos(t-s)-1}{2}.
	\]
	Since $\cos (t-s)-1\leq 0$, it is sufficient to require that $u(t)>0$.

	Finally, equation \eqref{eq:expgamma} can be written explicitly as 
	\[
	(1-z(s))u(t) > 2 \cos(t-s)- \cos t -1,
	\]
	where $z(s) = \frac{9}{8} \cos s - \frac{1}{8}\cos(3s )$ is the last component of $\gamma(s)$. For $s=0$ this is satisfied trivially, and for $s\neq 0$ we have
	\[
	u(t) > \frac{2 \cos(t-s)- \cos t -1}{1-z(s)}.
	\]
	The function on the right-hand side is continuous on $(0,2\pi]$ and goes to $-\infty$ when $s\to 0_+$. Hence it must attain a maximum on $(0,2\pi]$. We can set $u(t)$ to be positive and larger than this maximum. We conclude that any point on the curve $\alpha$ is an exposed face of $C$. 
\end{proof}

\begin{proposition}\label{prop:threedimf} The only two-dimensional faces of $C$ are the disk faces 
	\[
	\stdFaceC_\alpha:= \conv \alpha,\qquad \stdFaceC_\beta = \conv \beta.
	\]
	These faces are exposed.
\end{proposition}
\begin{proof}  We first show that $\stdFaceC_\alpha$ and $\stdFaceC_\beta$ are exposed faces of $C$. 
Observe that 
	\[
	\langle \alpha(t), (0,0,1)\rangle = 1 \quad \forall t \in [0,2\pi];
	\]
	\[
	\langle \beta(t), (0,0,1)\rangle = -1 <1 \quad \forall t \in [0,2\pi]; 
	\]
	\[
	\langle  \gamma(t), (0,0,1)\rangle  <1 \quad \forall t \in (0,2\pi).
	\]
	We deduce that the plane $H$ defined by  $z=1$ supports $C$, and that (invoking Proposition~\ref{prop:extreme}) $\ext C \cap H = \alpha$, hence, $\conv \alpha$ is an exposed face of $C$. The proof for $\stdFaceC_\beta$ and the plane $z = -1$ is analogous.

	To show that there are no other two-dimensional faces, assume that $\stdFaceC$ is a two-dimensional face of $C$. Then $\stdFaceC$ must contain at least three affinely independent points of $\ext C = \alpha\cup \beta\cup \gamma$ (see Proposition~\ref{prop:extreme}). In the case when at least two of these points belong to either $\alpha$ or $\beta$, the line segment connecting these two points intersects the relative interior of one of the disk faces, and hence the entire face $\stdFaceC$ must include this disk face, which means that the face $\stdFaceC$ coincides with either $\stdFaceC_\alpha$ or $\stdFaceC_\beta$. Therefore, for $\stdFaceC$ to be different to $\stdFaceC_\alpha$ or $\stdFaceC_\beta$ each of the curves $\alpha$ and $\beta$ can have at most one of these three affinely independent points. 
	
	Suppose that $\alpha$ and $\beta$ contain at least one point each, and consider the cylinder $\conv(\alpha \cup \beta)$ that is a subset of $C$. Note that since the interior of this cylinder is nonempty, the set $C$ is also three-dimensional.  If these points have a different projection onto the $xy$ plane, then the line segment connecting them intersects the interior of the aforementioned cylinder, hence, it intersects the interior of $C$, and the face $\stdFaceC$ has to be three-dimensional, a contradiction.

	If these two points correspond to the same value of the parameter $t$, then any supporting plane to $C$ that contains these two points must also be supporting to the cylinder. The only one such supporting plane that does not cut through the rest of the set $C$ is $x=2$, corresponding to the value $t=0$. This plane only contains two points of $\ext C$, $\alpha(0)= \gamma(0)$ and $\beta(0) = \gamma(2 \pi)$, hence, this plane can not be exposing a two-dimensional face. 
	
	We conclude that at most one of the three points lies on $\alpha\cup \beta$, and hence at least two different points must be on $\gamma\setminus (\alpha\cup \beta)$. Suppose that these points are $\gamma(t)$ and $\gamma(s)$, where $0< t<s< 2\pi$.  We will show that this arrangement is also impossible.
	
	Assume the contrary. Then $\gamma(t)$ and $\gamma(s)$ belong to some two-dimensional face $\stdFaceC\face C$, and there must be a plane exposing $\stdFaceC$; this plane must contain these two points. This plane must also contain the tangent lines $\gamma(t)+\R\gamma'(t)$  and $\gamma(s)+\R \gamma'(s)$. This is only possible if the vectors $\gamma(t) - \gamma(s), \gamma'(t), \gamma'(s)$ are linearly dependent. 
	
	Let 
	\[ M = \begin{bmatrix} \gamma(t) - \gamma(s) & \gamma'(t) & \gamma'(s)\end{bmatrix}.\]
	We would like to understand when this vanishes for $0<t<s < \pi$. After the change of variables $x = (s+t)/2$ and $y = (s-t)/2$,
	\[ \det(M) = -32\cos(y)\sin(x) \sin(y)^4\left[6+3\cos(2x)+\cos(2(x-y))+\cos(2y) + \cos(2(x+y))\right].\]
	Furthermore, we have that 
	\begin{align*}
	&	6+3\cos(2x) + \cos(2(x-y)) + \cos(2y) + \cos(2(x+y)) \\
	&= 2 + 4\cos(x)^2 + 6 \cos(x)^2\cos(y)^2 + 2\sin(x)^2\sin(y)^2\\
	& \geq 2\quad\textup{for all $x,y$.}
	\end{align*}
	Therefore, the only way that $\det(M)$ can vanish is if
	either $\cos(y) = 0$ or $\sin(x) = 0$ or $\sin(y)=0$.
	\begin{itemize}
		\item $\sin(y)=0$ if and only if $s-t$ is an integer multiple of $2\pi$. 
		Since $s,t\in (0,\pi)$, this is impossible. 
		\item $\sin(x)=0$ if and only if $s+t$ is an integer multiple of $2\pi$. Since $s,t\in (0,\pi)$, 
		this is impossible. 
		\item $\cos(y)=0$ if and only if $s-t$ is an odd multiple of $\pi$. Since $s,t\in (0,\pi)$, this 
		is impossible. 
	\end{itemize}

\end{proof}

\begin{proposition}\label{prop:fexp} The cone $\stdCone = \cone (C\times \{1\})$ is facially exposed.
\end{proposition}
\begin{proof} 
	If $C$ is facially exposed, then $\stdCone$ is also facially exposed (e.g., see \cite[Proposition~3.2]{Vera}). Therefore it is sufficient to demonstrate that $C$ is facially exposed. 
	
	We know from Propositions~\ref{prop:extreme} and \ref{prop:threedimf} that all zero-dimensional and two\-/dimensional faces of the set $C$ are exposed. If there is a one-dimensional face that is not exposed, then it must be a subface of some two-dimensional face that is exposed, see Proposition~\ref{prop:exp_face}. This is impossible, since the only two-dimensional faces of $C$ are disks by Proposition~\ref{prop:threedimf}, and so do not have one-dimensional subfaces. We conclude that all one-dimensional faces are exposed. 	
\end{proof}

The next result will be useful in what follows. It was proved within \cite[Theorem~3]{pataki}. Recall that a face $\stdFace \face \stdCone$ is \emph{properly minimal} if it does not coincide with the lineality space of $\stdFace$ and does not have any subfaces that strictly contain the lineality space. For instance, properly minimal faces of a pointed cone are its extreme rays. 

\begin{theorem}[Pataki criterion]\label{thm:pataki} Let $\stdCone\subseteq \R^n$ be a closed convex cone. If $\stdCone$ is facially exposed, and for some face $\stdFace \face \stdCone$ all properly minimal faces of $\stdFace^*$ are exposed, then $\stdCone^*+\stdFace^\perp$ is closed.
\end{theorem}

\begin{proof}[Proof of Proposition~\ref{prop:fdc}] 
	
	The cone  $\stdCone$ is facially exposed by Proposition~\ref{prop:fexp}. Since every face $\stdFace$ of $\stdCone$ of dimension 2 and less is polyhedral (as is the case for any closed convex cone), we have from Theorem~\ref{thm:pataki} that $\stdFace^\perp+\stdCone^*$ is closed for all such faces. To finish the proof of facial dual completeness, it is sufficient to demonstrate that $\stdFace^\perp+\stdCone^*$ is closed for all three-dimensional faces $\stdFace$ of $\stdCone$. 
	
	We know from Proposition~\ref{prop:threedimf} that the only three-dimensional faces of $\stdCone$ are the lifts of the disk faces $\conv \alpha$ and $\conv \beta$. We will show that $\stdFace^\perp+\stdCone^*$ is closed for $\stdFace = \cone ((\conv \alpha)\times \{1\})$. The proof for the second three-dimensional face is analogous due to symmetry.

	Let $\tilde \stdCone = \cone \{\tilde C\times \{1\}\}$, where $\tilde C = \conv (\alpha\cup \beta)$. Since $\tilde \stdCone\subseteq \stdCone$, we have $\stdCone^* \subseteq \tilde \stdCone^*$ and 
	\[
	\stdCone^* + \stdFace^\perp \subseteq \tilde \stdCone^* + \stdFace^\perp.
	\]
	We will show that $\tilde \stdCone^*+\stdFace^\perp$ is closed and that $\stdCone^* + \stdFace^\perp \supseteq \tilde \stdCone^* + \stdFace^\perp$ (and hence $\stdCone^* + \stdFace^\perp = \tilde \stdCone^* + \stdFace^\perp$ is closed).

	We note that $\tilde{\stdCone} = \{(a,b,c,t)\:\; -t\leq c \leq t,\;\sqrt{a^2+b^2}\leq t\}$.
	Consider the face
	\[ \stdFace = \{(a,b,c,t)\in \tilde{\stdCone}\;:\; c=t\} = 
	\{(a,b,t,t)\;:\; \sqrt{a^2+b^2}\leq t\}\]
	and note that $\stdFace^\perp = \textup{span}\{(0,0,1,-1)\}$.
	
	The dual cone of $\tilde{\stdCone}$ is  
	\[ \tilde{\stdCone}^* = \{(x,y,z,w)\;:\; \sqrt{x^2+y^2} + |z| \leq w\}.\]
	We can then directly compute $\tilde{\stdCone}^*+\stdFace^\perp$ as 
	\begin{align*}
	\tilde{\stdCone}^*+\stdFace^\perp & = \{(x,y,z,w)\;:\;\exists \mu\in \RR\;\;\textup{s.t.}\;\;
	\sqrt{x^2+y^2} + |z-\mu| \leq w+\mu\}\\
	& = \{(x,y,z,w)\;:\; \exists \mu\in \RR\;\;\textup{s.t.}\;\;\sqrt{x^2+y^2} \leq (w+\mu)+(z-\mu),\\
		&\qquad\qquad\qquad\qquad\qquad \qquad\qquad\qquad\qquad\quad\sqrt{x^2+y^2}\leq (w+\mu)-(z-\mu)\}\\
	& = \{(x,y,z,w)\;:\; \exists \mu\in \RR\;\;\textup{s.t.}\;\;\sqrt{x^2+y^2} \leq z+w,\; \sqrt{x^2+y^2}\leq w-z+2\mu\}\\
	& = \{(x,y,z,w)\;:\; \sqrt{x^2+y^2} \leq z+w\}.
	\end{align*}
	This is the preimage of the second-order cone (which is closed) 
	under the linear map $(x,y,z,w) \mapsto (x,y,z+w)$ and so is a closed set (another way to see this is to observe that $\tilde{\stdCone}$ is a spectrahedron and $\stdFace$ is a face of $\tilde{\stdCone}$). It follows that	$\tilde{\stdCone}^* + \stdFace^\perp$ is closed.

	We now aim to show that $\tilde{\stdCone}^* + \stdFace^\perp \subseteq \stdCone^* + \stdFace^\perp$. We do this by identifying
	a particular set $E\subseteq \stdCone^*$ and showing that an arbitrary element of the boundary of $\tilde{\stdCone}^*+\stdFace^\perp$ 
	is contained in $E + \stdFace^\perp$. 
	
	Let $p(t) = (\cos(t),\sin(t),u(t))$ be the choice of exposing hyperplane for the point $(\cos(t),\sin(t),1)$
	from Proposition~\ref{prop:extreme}. We know that 
	\[\langle p(t),x\rangle \leq
	\langle p(t),(\cos(t),\sin(t),1)\rangle = 
	1+u(t)\;\;\textup{for all $x\in C$ and all $t\in [0,2\pi]$}.\]
	It then follows that 
	$(-\cos(t),-\sin(t),-u(t),1+u(t))\in \stdCone^*$.
	Our aim is to show that any element of $\tilde{\stdCone}^* + \stdFace^\perp$ can be expressed in the form
	\[ \alpha(-\cos(t),-\sin(t),-u(t),1 + u(t)) + \beta (0,0,1,-1)\]
	for some $t\in [0,2\pi]$, some $\alpha \geq 0$, and some $\beta\in \RR$. 
	Let $(x,y,z,w)$ be an arbitrary boundary point of $\tilde{\stdCone}^* + \stdFace^\perp$, in other words, an arbitrary point
	satisfying $\sqrt{x^2+y^2}= z+w$. Letting $\alpha = z+w = \sqrt{x^2+y^2}\geq 0$, we can find $t$ such that 
	$(x,y) = \alpha(-\cos(t),-\sin(t))$. Furthermore, given that particular $t$, we can write
	\[ \begin{bmatrix} z\\w\end{bmatrix} = 
	(z+w)\begin{bmatrix} -u(t)\\1+u(t)\end{bmatrix} + 
	(z + (z+w)u(t))\begin{bmatrix} 1\\-1\end{bmatrix} = \alpha\begin{bmatrix} -u(t)\\1+u(t)\end{bmatrix} + \beta\begin{bmatrix}1\\-1\end{bmatrix}\]
	where $\beta = (z+(z+w)u(t))$. 
	Overall, then, we see that $(x,y,z,w)\in \stdCone^* + \stdFace^\perp$.
	Since $(x,y,z,w)$ is an arbitrary element of the boundary of $\tilde{\stdCone}^*+\stdFace^\perp$, 
	we have showed that the boundary of $\tilde{\stdCone}^*+\stdFace^\perp$ is contained in  $\stdCone^* + \stdFace^\perp$.
	By convexity, it follows that $\tilde{\stdCone}^*+\stdFace^\perp \subseteq \stdCone^* + \stdFace^\perp$. 	
\end{proof}

\subsection{Non-amenability}

\begin{proof}[Proof of Proposition~\ref{prop:notamen}] From Proposition~\ref{prop:am_int} the intersection of two amenable sets is amenable. Our goal is to show that the set $C$ is not amenable. Since $C$ is the intersection of an affine subspace with $\stdCone$, this shows $\stdCone$ is not amenable. 
	
	By Proposition~\ref{prop:threedimf}, the disk $\stdFaceC = \conv \alpha$ is a face of $C$. We will apply the definition of amenability (Definition~\ref{def:am}) to the bounded set $B = \{(x,y,1): (x-1)^2 +y^2 \leq 1\}$.
	
	Let $w(t) = (2\cos(2t)-1,2\sin(2t),1)$ which lies in $\aff (\stdFaceC) \cap B$ for sufficiently small $t \geq 0$. 
	It is enough to show that 
	\[\frac{\dist(w(t),C)^2}{\dist(w(t),\stdFaceC)^2} \rightarrow 0
	\quad \text{as} \quad t\rightarrow 0^{+}.
	\] 
	 Now, using the Taylor expansion we have
	\begin{align*} 
	\dist(w(t),C)^2 & \leq \dist(w(t),\gamma(t))^2\\
	& = (1-(9/8)\cos(t) + (1/8)\cos(3t))^2\\
	& = \frac{9}{64}t^8 + O(t^{10}).
	\end{align*}
	On the other hand, noting that $(2\cos(2t)-1)^2+(2\sin(2t))^2 = 5-4\cos(2t)$, we see that
	\begin{align*}
	\dist(w(t),\stdFaceC)^2 & = \min_{u^2+ v^2 = 1} (u-(2\cos(2t)-1))^2 + (v-2\sin(2t))^2\\
	& = (1-\sqrt{5-4\cos(2t)})^2= 16t^4 + O(t^6).
	\end{align*}
	It then follows that 
	\[ \lim_{t\rightarrow 0}\frac{\dist(w(t),C)^2}{\dist(w(t),\stdFaceC)^2} \leq  \lim_{t\rightarrow 0}\frac{\frac{9}{64}t^8 + O(t^{10})}{16t^4 + O(t^6)} = 0.\]

	We deduce that there is no constant $\kappa $ satisfying the definition of amenability for the set $C$ and its face $\stdFaceC$, and hence by the earlier observation the cone $\stdCone$ is not amenable.
\end{proof}

\section{Amenability and projectionally exposed cones}\label{sec:amenproj}
The current situation is that the different notions of exposedness described in Proposition~\ref{prop:imp} all coincide in dimension three and there are examples in dimension four of facially exposed cones that are not nice \cite{Vera} and nice cones that are not amenable (Section~\ref{sec:nicenotamen}).

The next natural question would be to clarify the relationship between amenability and projectional exposedness. 
In this section, we will see, however, that if $\dim \stdCone \leq 4$, amenability implies projectional exposedness, so any counter-example can only appear in dimension five or more.

We start with the following technical criterion for projectional exposedness by 
Sung and Tam. For $\stdFace \face \stdCone$, we define its \emph{conjugate face} as $\stdFace^{\Delta} \coloneqq \stdCone^* \cap \stdFace^\perp$.
\begin{theorem}[Sung and Tam's criterion, item~(a) of Theorem~3.2 in \cite{ST90}]\label{thm:st}
Let $\stdCone$ be a pointed full-dimensional closed convex cone and $\stdFace \face \stdCone$ be a face of codimension $1$. 
Let $w$ be such that  $\stdFace ^ \Delta = \{\alpha w \mid \alpha \geq 0\}$. Then, 
$\stdFace$ is a projectionally exposed face if and only if $w$ is not the limit of a convergent sequence  $\{w_k\} \subseteq \stdCone^*$ such that 
the $w_k$ generate extreme rays  distinct from  $\stdFace ^ \Delta$.
\end{theorem}

\begin{theorem}\label{thm:co_proj}
Let $\stdCone$ be a full-dimensional pointed closed convex cone. 
If $\stdFace \face \stdCone$ is an amenable face of 
codimension $1$, then $\stdFace$ is projectionally exposed.
\end{theorem}
\begin{proof}
$\stdCone$ and $\lspan \stdFace$ are boundedly linearly regular by Proposition~\ref{prop:am_eq}. By \cite[Theorem~10]{BBL99}, this means that the so-called \emph{property (G)} holds for $- \stdCone^*$ and  $\stdFace^\perp$.
That is, denoting the unit ball in $\ambSpace$ by $U = \{x \mid \norm{x}\leq 1\}$, there exists $\alpha > 0$ such that
\[
U \cap (- \stdCone^* +   \stdFace^\perp) \subseteq 
\alpha (U \cap(- \stdCone^*)  + U\cap\stdFace^\perp ).
\]
This implies that 
\begin{equation}\label{eq:am_g}
U \cap (\stdCone^* +   \stdFace^\perp) \subseteq 
\alpha (U \cap\stdCone^*  + U\cap\stdFace^\perp ).
\end{equation} 
Let $w$ be such that $\norm{w} = 1$ and $\stdFace^\Delta =  \{\beta w \mid \beta        \geq 0\}$ and suppose that $\stdFace$ is not projectionally exposed. Then, by Theorem~\ref{thm:st},
there exists a sequence  $\{w_k\} \subseteq \stdCone^*$ such that $w_k \to w$
and the $w_k$ generate extreme rays that are all distinct from  $\stdFace ^ \Delta$.

Because $\stdFace$ has codimension $1$, $\stdFace^\perp$  is generated by $w$. Therefore,
\eqref{eq:am_g} implies
\begin{equation}\label{eq:am_g2}
U \cap (\stdCone^* +   \stdFace^\perp)  \subseteq 
\alpha (U \cap\stdCone^*  + \{\beta w \mid \beta \in [-1,1] \} ).
\end{equation} 
Since $2\alpha (w_k - w) \in \stdCone^* + \stdFace^\perp$
and $w_k - w \to 0$, in view of \eqref{eq:am_g2}, for sufficiently large $k$, there exists $y_k \in U \cap \stdCone^*$ and $\beta _k \in [-1,1]$ such that $
2\alpha (w_k -w) = \alpha\left({y_k} + \beta _k {w}\right)$. 
Equivalently,
 \begin{equation}\label{eq:wkyk}
 w_k =  \frac{y_k}{2} + \left(1+\frac{\beta _k}{2}\right)w.
 \end{equation}
Since $\beta_k \in [-1,1]$, we have $\left(1-\frac{\beta _k}{2}\right) > 0$. So, 
\eqref{eq:wkyk} implies that $w$ lies in the extreme ray generated by $w_k$ (for sufficiently large $k$), which is a contradiction.
Therefore, the sequence $\{w_k\}$ cannot exist and 
$\stdFace$ must be projectionally exposed.
\end{proof}

Next, we recall that a closed convex cone $\stdCone$ is projectionally exposed if and only if 
its ``pointed component'' $\stdCone \cap \lineality \stdCone^\perp$ is projectionally exposed, see \cite[Lemma~2.4]{ST90} and its proof.
 We also have the following well-known fact.
\begin{lemma}[Folklore\footnote{This fact is referenced in several articles but it is not completely trivial to find a proof. Barker used the condition $\stdFace = \stdFace^{\Delta\Delta}$ as the definition of exposed face, see \cite[Definition 2.A.9]{Ba81} and mentioned that this coincides with the definition using exposing hyperplanes. Barker's definition was adopted in some works on convex cones in 80s, e.g., \cite{Tam85,ST90}. We also noted that Br{\o}ndsted's book is sometimes mentioned as reference for the proof, but, actually, the proof is given for compact convex sets only, see \cite[Theorem~6.7]{Br83}.}]\label{lem:exposed}
$\stdFace \face \stdCone$ is facially exposed if and only if $\stdFace = \stdFace^{\Delta\Delta}$.
\end{lemma}
\begin{proof}
We note that $\stdFace^\Delta$ is always an exposed face of $\stdCone^*$ because if $x \in \reInt \stdFace$, we have $\stdFace^\Delta = \stdCone^*\cap \{x\}^\perp $. Therefore, if  $\stdFace = \stdFace^{\Delta\Delta}$, then $\stdFace$ is facially exposed. Conversely, suppose $\stdFace$ is facially exposed and let $s \in \stdCone^*$ be such that $\stdFace = \stdCone \cap \{s\}^\perp$. Then, 
$s \in \stdFace^\Delta$. Let $\hat \stdFace$ be the minimal face of $\stdCone^*$ containing $s$, we have $s \in \reInt \hat \stdFace$ and $\hat \stdFace \face \stdFace^\Delta$. Therefore, $\stdFace^{\Delta\Delta} \face \hat \stdFace^\Delta = \stdFace$. Since, we always have $\stdFace \face \stdFace^{\Delta\Delta}$ for any face, this shows that $\stdFace^{\Delta\Delta} = \stdFace$.	
\end{proof}

\begin{corollary}\label{cor:proj}
If $\stdCone \subseteq \ambSpace$ is closed convex cone of 
dimension $\dim \stdCone \leq 4$, then 
$\stdCone$ is amenable if and only if 
it is projectionally exposed.
\end{corollary}
\begin{proof}
By Proposition~\ref{prop:imp}, projectionally exposed cones are amenable and the converse holds if $\dim \stdCone \leq 3$. So, we assume that $\dim \stdCone = 4$ and that $\stdCone$ is amenable. Then, span $\stdCone$ is linearly isomorphic to $\Re^4$ and 
the same isomorphism shows that 
$\stdCone$ is linearly isomorphic to a cone $\hat \stdCone \subseteq \Re^4$ which is amenable (by item~\eqrefit{prop:am_int:3} of Proposition~\ref{prop:am_int}) and full-dimensional. 
Let $\widetilde{\stdCone} \coloneqq \hat \stdCone \cap (\lineality \hat \stdCone)^\perp$. By item~\eqrefit{prop:am_int:5} of Proposition~\ref{prop:am_int}, $\widetilde{\stdCone}$ is amenable. Furthermore, $\stdCone$ is projectionally exposed if and only if $\widetilde{\stdCone}$ is projectionally exposed, because  $\hat \stdCone$ and $\stdCone$ are linearly isomorphic and $\widetilde{\stdCone}$ is the pointed component of $\hat \stdCone$.

$\widetilde{\stdCone}$ is an amenable pointed full-dimensional closed convex cone and we will show that it is also projectionally exposed. For that, 
let $\stdFace \face \widetilde{\stdCone}$. If $\stdFace = \{0\}$ or $\stdFace = \widetilde{\stdCone}$, 
then the zero map and the identity map are projections that map $\widetilde{\stdCone}$ onto $\{0\}$ and $\widetilde{\stdCone}$, respectively. 
Next, we consider three cases.

\fbox{Case 1: $\dim \stdFace = 1$}
In this case, $\stdFace$ can be written 
as $\stdFace = \{\alpha x \mid \alpha \geq 0 \}$, for some nonzero $x \in \widetilde{\stdCone}$. Let $z \in \widetilde{\stdCone}^*$ be such that $\inProd{x}{z} = 1$. At least one such $z$ must exist because otherwise we would have $x \in \widetilde{\stdCone}^{*\perp} = \lineality \widetilde{\stdCone} = \{0\}$.
With that, the projection defined by $P = xz^T$
satisfies $P(\widetilde{\stdCone}) = \stdFace$.

\fbox{Case 2: $\dim \stdFace = 2$}
The argument is essentially the same as \cite[Corollary~4.8]{ST90}.
$\stdFace$ can be written 
as $\stdFace = \{\alpha x + \beta y  \mid \alpha \geq 0, \beta \geq 0 \}$, where $x$ and $y$ generate different extreme rays of $\widetilde{\stdCone}$ denoted respectively by $\stdFace_x$ and $\stdFace _y$. 

Since $\widetilde{\stdCone}$ is facially exposed, 
$\stdFace_x ^ \Delta = \widetilde{\stdCone}^* \cap \{x\}^\perp$ and 
$\stdFace _y ^ \Delta = \widetilde{\stdCone}^* \cap \{y\}^\perp$ must be different faces of $\widetilde{\stdCone}^*$  by  Lemma~\ref{lem:exposed}.
Furthermore, $\stdFace_x ^ \Delta$ and $\stdFace_y ^ \Delta$ are not contained in each other.
In particular, we can find $z_1 \in \stdFace_x ^ \Delta$ that does not belong to $\stdFace _y ^ \Delta $ and, also, $z_2 \in \stdFace_y ^ \Delta$ that does not belong to $\stdFace _x ^ \Delta $. Rescaling $z_1, z_2$ if necessary, we have 
\[
\inProd{x}{z_1} = 0, \inProd{x}{z_2} = 1, \inProd{y}{z_1} = 1, \inProd{y}{z_2} = 0.
\]
Therefore, the projection defined by $P = xz_2^T + yz_1^T$ maps $\widetilde{\stdCone}$ to $\stdFace$.

\fbox{Case 3: $\dim \stdFace = 3$}
Follows by Theorem~\ref{thm:co_proj}.
\end{proof}

\section{Open Problems}

In this section we outline some open questions related to the geometry of convex cones, motivated by our study of amenability.

\subsection{Characterisation of amenability and niceness via slices}

The polar of the slice $C$ of $\stdCone$ studied in Section~\ref{sec:nicenotamen} is shown in Fig.~\ref{fig:Cpolar}.
\begin{figure}[ht]
	\centering
	\includegraphics[width=0.45\textwidth]{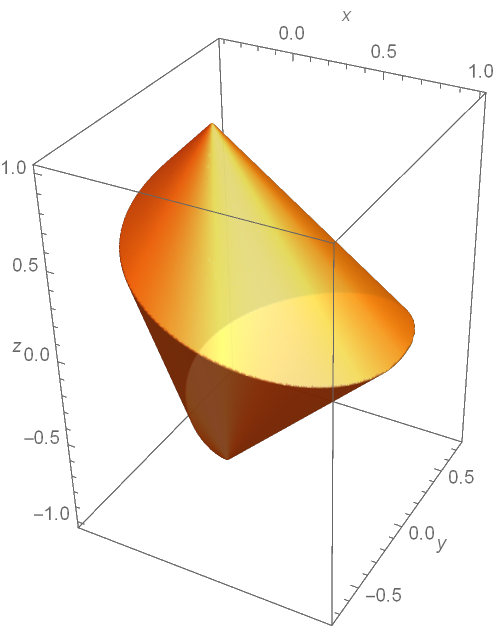}\quad 
	\includegraphics[width=0.45\textwidth]{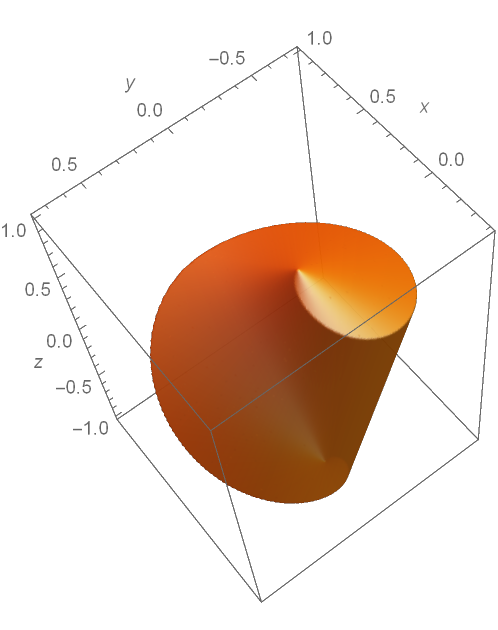}\quad 
	\caption{The polar of the set $C$ from Section~\ref{sec:nicenotamen}.}
	\label{fig:Cpolar}
\end{figure} 
This set appears to have a peculiar arrangement of faces adjacent to the two tips (dual to the disk faces of $C$). The cone of feasible directions at each tip appears to be closed, however, the set lacks the exactness of tangent approximation (ETA) property (see \cite{MRY15}) at these points. In other words, there is no neighbourhood in which the set coincides with its tangent. (Also note that there are sequences of extreme points converging to the tips, and hence the conditions of Theorem~\ref{thm:st} are not satisfied: we immediately see from this image that the relevant cone is not projectionally exposed.) On the other hand, it appears---based on this and other known four-dimensional examples---that niceness in $\R^4$ corresponds to the closedness of the set of feasible directions at extreme points of the polar that are dual to the two-dimensional faces of the primal slice. We wonder if it is possible to obtain a general characterisation of (and distinguish between) niceness, amenability and projectional exposedness using these kinds of tangential properties pertaining to the polars of slices.

\subsection{Projectionally exposed cones} 

It was shown in \cite{L19} that projectionally exposed cones are amenable, however we do not know whether the converse is false. We failed to construct an example of an amenable cone that is not projectionally exposed. In view of Corollary~\ref{cor:proj}, if such an example exists it must be of dimension at least five.

In addition, as seen in Table~\ref{tab:table1}, it is unknown whether homogeneous cones are projectionally exposed and whether projectional exposedness is preserved under intersections. 
We note that a positive answer to the latter would imply projectional exposedness of all spectrahedral cones, including all the homogeneous cones. 

\subsection{Tangentially and strongly tangentially exposed cones}

It was shown in \cite{RT} that necessary and sufficient conditions for niceness can be formulated using a yet another strengthening of the notion of facial exposedness. Specifically, if a cone $\stdCone$ is nice, then the tangent cone to every face of $\stdCone$ is the intersection of the span of this face with the tangent to the entire cone $\stdCone$ (this condition is called \emph{tangential exposure}). If this condition is satisfied for all tangent cones of $\stdCone$, and recursively for all tangents of tangents, then the cone is nice (this condition is called strong tangential exposure, and the higher-order tangents are called \emph{lexicographic tangents}). It is unknown what is the relationship between these tangential conditions and the notions of amenability and projectional exposure.

\section*{Acknowledgements}
We thank the referees and the associate editor for their  comments, which helped to improve the paper.

We are grateful to the School of Mathematics and Statistics at UNSW Sydney for providing financial assistance via a startup research grant that helped fund the meeting of all three coauthors in Sydney that initiated this project. 

Vera Roshchina is grateful to the Australian Research Council for continuing financial support that contributed to the successful execution of this work. 

Bruno F.\ Louren\c{c}o is grateful for the support of the JSPS through the Grant-in-Aid for Early-Career Scientists 19K20217 and the Grant-in-Aid for Scientific Research (B)18H03206 and 21H03398.

James Saunderson is the recipient of an Australian Research Council Discovery Early Career Researcher Award (project number DE210101056) funded by the Australian Government.
	
\bibliographystyle{abbrvurl}
\bibliography{refs}
\end{document}